\documentclass[12pt,a4paper]{elsarticle}
\usepackage[utf8]{inputenc}
\usepackage{xcolor}
\usepackage{mathrsfs}
\usepackage{mathtools}
\usepackage{amsmath}
\usepackage{amsthm}
\usepackage{amssymb}
\usepackage{graphicx}
\usepackage[normalem]{ulem}
\usepackage[hmargin=2cm,bmargin=2.5cm,tmargin=3cm]{geometry}

\makeatletter

\providecommand{\tabularnewline}{\\}

\numberwithin{equation}{section}
\numberwithin{figure}{section}

\newtheorem{theorem}{Theorem}[section]
\newtheorem{assum}[theorem]{Assumptions}

\newtheorem{lem}[theorem]{Lemma}
\newtheorem{cor}[theorem]{Corollary}
\newtheorem{conj}[theorem]{Conjecture}
\newtheorem{thm}[theorem]{Theorem}

\theoremstyle{definition}
\newtheorem{defn}[theorem]{Definition}
\newtheorem{rem}[theorem]{Remark}

\DeclareMathOperator{\trhb}{tr_{_{\, HB}}}
\DeclareMathOperator{\diag}{diag}
\newcommand{\N}{\mathbb{N}}

\@ifundefined{showcaptionsetup}{}{%
 \PassOptionsToPackage{caption=false}{subfig}}
\usepackage{subfig}
\makeatother

\begin{document}

\begin{frontmatter}{}

\title{Hubs-biased resistance distances on graphs and networks}

\author{Ernesto Estrada$^{1,2,3}$ and Delio Mugnolo$^{4}$}

\address{$^{1}$Institute of Mathematics and Applications, University of Zaragoza,
Pedro Cerbuna 12, Zaragoza 50009, Spain; $^{2}$ARAID Foundation,
Government of Aragon, Spain. $^{3}$Institute for Cross-Disciplinary
Physics and Complex Systems (IFISC, UIB-CSIC), Campus Universitat
de les Illes Balears E-07122, Palma de Mallorca, Spain. $^{4}$Lehrgebiet
Analysis, Fakult{\"a}t Mathematik und Informatik, FernUniversit{\"a}t in Hagen,
D-58084 Hagen, Germany; }

\begin{abstract}
We define and study two new kinds of ``effective resistances'' based
on hubs-biased  --  hubs-repelling and hubs-attracting -- models of navigating
a graph/net\-work. We prove that these effective resistances are squared
Euclidean distances between the vertices of a graph. They can be expressed
in terms of the Moore--Penrose pseudoinverse of the hubs-biased Laplacian
matrices of the graph. We define the analogous of the Kirchhoff indices
of the graph based of these resistance distances. We prove several
results for the new resistance distances and the Kirchhoff indices
based on spectral properties of the corresponding Laplacians. After
an intensive computational search we conjecture that the Kirchhoff
index based on the hubs-repelling resistance distance is not smaller
than that based on the standard resistance distance, and that the
last is not smaller than the one based on the hubs-attracting resistance
distance. We also observe that in real-world brain and neural systems
the efficiency of standard random walk processes is as high as that
of hubs-attracting schemes. On the contrary, infrastructures and modular
software networks seem to be designed to be navigated by using their
hubs.

\end{abstract}

\begin{keyword}
graph Laplacians; resistance distances; spectral properties; algebraic connectivity; complex networks

\medskip{}

Corresponding author: Ernesto Estrada; email: estrada@ifisc.uib-csic.es
\end{keyword}

\end{frontmatter}{}

\section{Introduction}

Random walk and diffusive models are ubiquitous in mathematics, physics,
biology and social sciences, in particular when the random walker
moves through the vertices and edges of a graph $G=\left(V,E\right)$
\cite{random walks_1,random_walk_1.5,random walks_2,diffusion book,diffusion on networks}.
In this scenario a random walker at the vertex $j\in V$ of $G$ at
time $t$ can move to any of the nearest neighbors of $j$ with equal
probability at time $t+1$ \cite{random walks_1}. That is, if as
illustrated in Fig. \ref{models}(a) the vertex $j$ has three nearest
neighbors $\left\{ k,l,m\right\} $ the random walker can move to
any of them with probability $p_{jm}=p_{jl}=p_{jk}=k_{j}^{-1}$, where
$k_{j}$ is the degree of $j$. We can figure out situations in which
the movement of the random walker at a given position is facilitated
by the large degree of any of its nearest neighbors. 
{Here we use a relaxed definition of the term ``hub''. Although this term is used in network theory for those nodes of exceptionally large degree, we use it here in the following way: If there are two connected vertices of different degree, we call a ``hub'' the one of larger degree. This is formally defined later on in the paper.}

Then, let us suppose that there are situations in which the probability
that the random walker moves to a nearest neighbor of $j$ at time
$t+1$ increases with the degree of the nearest neighbor. This is
illustrated in Fig. \ref{models}(b) where $k_{i}>k_{l}>k_{m}$, and
consequently $p_{jm}<p_{jl}<p_{j{i}}$. We will refer hereafter to this
scenario as the \textit{``hubs-attracting}'' one. Another possibility
is that the random walker is repelled by large degree vertices, such as
for $k_{i}>k_{l}>k_{m}$, we have $p_{jm}>p_{jl}>p_{j{i}}$ as illustrated
in Fig. \ref{models}(c). We will refer to this model as the ``\textit{hubs-repelling}''
one. These scenarios could be relevant in the context of real-world
complex networks where vertices represent the entities of the system
and the edges the interrelation between them \cite{Boccaletti_review,Newman_review,Estrada book}.
An example of hubs-repelling strategies of navigation are some of
the diffusive processes in the brain where there is a high energetic
cost for navigating through the hubs of the system \cite{brain metabolism,brain navigation}.
Hubs-attracting mechanisms could be exhibited, for instance, by diffusive
epidemic processes in which hubs are major attractors of the disease and can be considered at high risk of contagion by spreaders in the network \cite{epidemics}.

\begin{figure}
\begin{centering}
\subfloat[]{\begin{centering}
\includegraphics[width=0.3\textwidth]{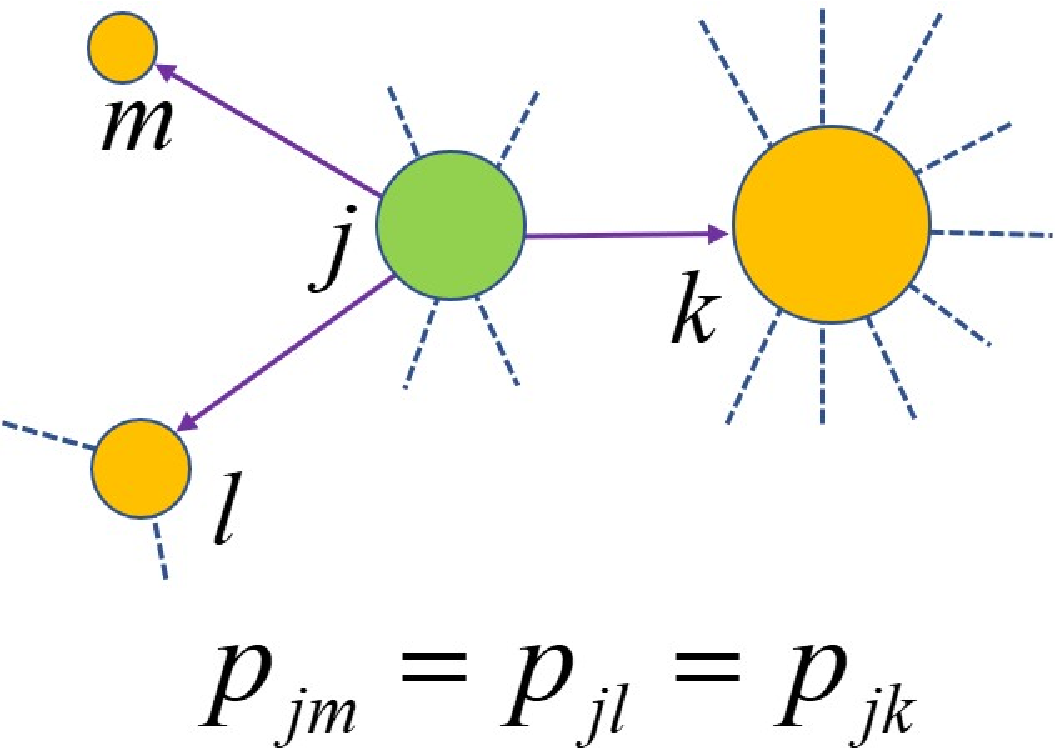}
\par\end{centering}

}\subfloat[]{\begin{centering}
\includegraphics[width=0.3\textwidth]{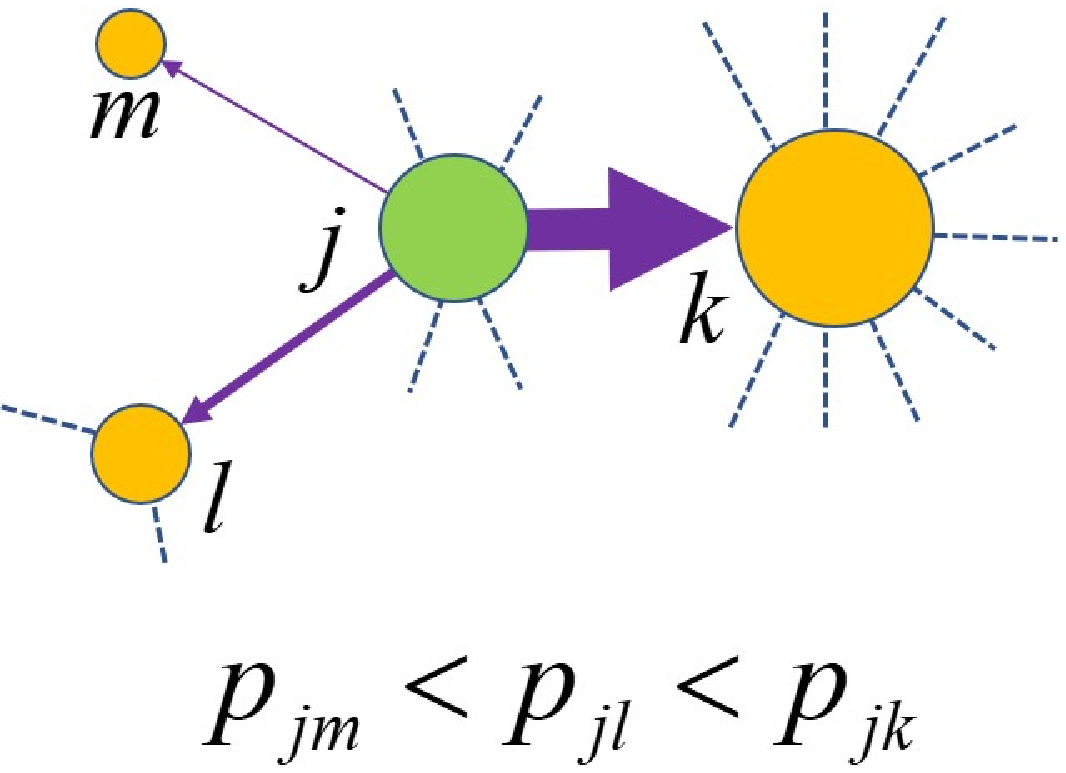}
\par\end{centering}
}\subfloat[]{\begin{centering}
\includegraphics[width=0.3\textwidth]{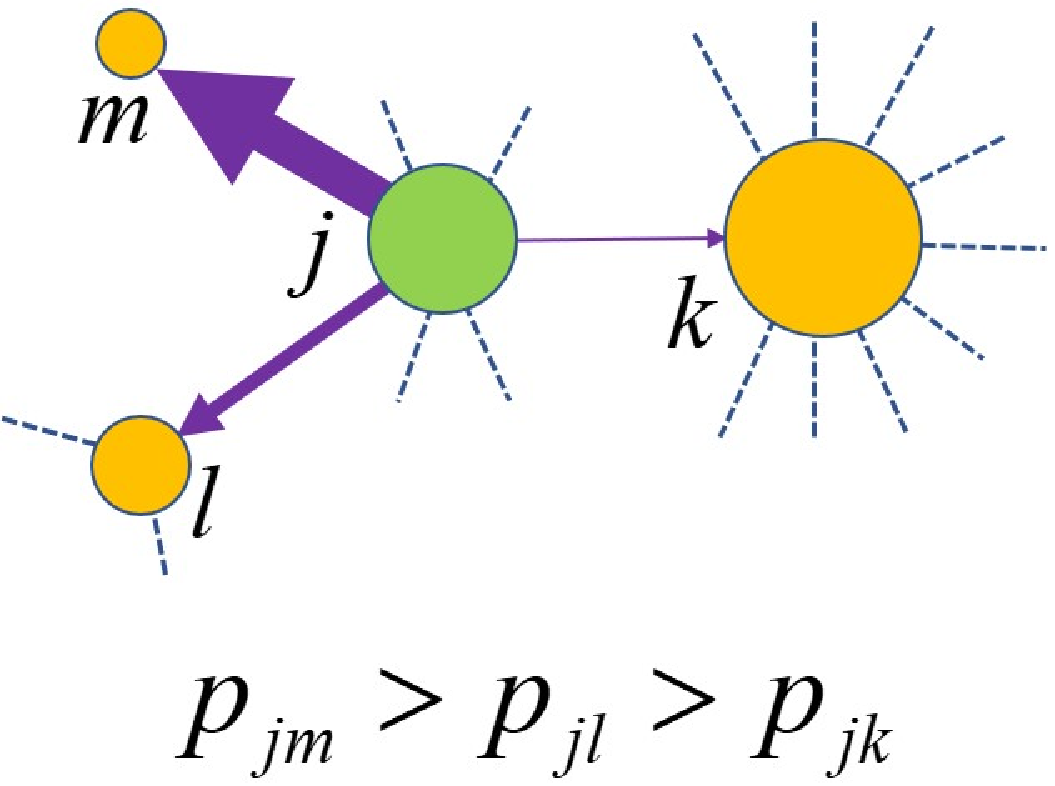}
\par\end{centering}
}
\par\end{centering}
\caption{Schematic illustration of a normal (a), hubs-attracting (b) and hubs-repelling
schemes of a particle hopping on a network.}

\label{models}
\end{figure}

From a mathematical perspective, one of the most important aspects
of the model of random walks on graphs is its connection with the
graph Laplacian matrix \cite{Laplacian_1,Laplacian_2,Laplacian_3,Laplacian_4,Laplacian_5}
and with the concept of resistance distance \cite{resistance_0,resistance_1,resistance_2,resistance_3,resistance_4,resistance_5}.
In a simple graph the resistance distance is the effective resistance
between two vertices $v$ and $w$, which measures the resistance of
the total system when a voltage is connected across $v$ and $w$.
Klein and Randi\'{c} \cite{resistance_1} proved that the effective
resistance is a squared Euclidean distance between the two vertices of
the graph, which can be obtained from the Moore--Penrose pseudoinverse
of the graph Laplacian. It is also known that the commute time between
two vertices $v$ and $w$ of a random walker \cite{resistance_0,resistance_2,resistance_4,resistance_5},
i.e., the number of steps of a random walker starting at $v$, before
arriving at $w$ and returning to $v$ again, is proportional to the
resistance distance between the two vertices.

Strategies for avoiding large/small degree vertices in random-walk processes
on graphs/networks have been proposed in the literature under the
general umbrella of `degree-biased random walks' \cite{biased_1,biased_2,biased_3,biased_4}.
However, in the current work we go beyond the random walk formulation
of the problem and express it in terms of hubs-biased Laplacians \cite{hubs-repelling,hubs-attracting}
and the corresponding resistance distance matrices on graphs. Thus,
we focus here on the algebraic properties of these resistance distance
matrices. We note in passing that the current concept has nothing
in common with the so-called ``degree-resistance'' between a pair
of vertices, which is nothing else that the resistance distance multiplied
by the difference of degrees of the vertices forming the corresponding
edge \cite{degree-resistance}.

Due to the relation between resistance distance and commute time of
random walks on graphs, we study here the efficiency of hubs attracting/repelling
diffusive processes on graphs. In closing, in this work {we
define the hubs-biased resistance distances between pairs of vertices
of a simple graph and study their main spectral properties. We also
propose analogues of the Kirchhoff index \cite{resistance_1,Kirchhoff_1,Kirchhoff_2,Kirchhoff_3},
the semi-sum of all resistance distances in the graph, for the hubs-biased
resistances. We report here several bounds for the two new kinds of
resistance distances as well as for the corresponding Kirchhoff indices.
{
Finally, we study the commute time of the hubs-attracting random walks and analyze their relative improvement over the standard one.}
We observe that certain classes of real-world networks, such as brain/neuronal networks and electronic circuits, have normal random walks as efficient as the hubs-attracting one, while others, like infrastructural networks,
can reduce their average commuting times by 300\% by using the hubs-attracting
mechanism.}

\section{Preliminaries}

{
In this article we consider simple weighted graphs.
We will always impose the following throughout.

\begin{assum}\label{assum:basic}
$G$ is a weighted graph with vertex set $V$ and edge set $E$. The underlying unweighted graph $(V,E)$ is simple and finite: we denote by $n,m\in \mathbb N$ the number of vertices and edges, respectively. To avoid trivialities, we also assume $n\ge 2$ and that no vertex is isolated. We also denote by $\mathcal C$ its number of connected components.

 Additionally,  each edge is assigned a weight by means of a surjective mapping $\varphi:E\rightarrow W$, with $W\subset (0,\infty)$.
\end{assum}
In the following we use interchangeably the terms graphs and networks.
}

%
%
Let $A$ be the adjacency matrix of the (weighted) graph $G$ and
let $k_{i}$ denote the degree of the vertex $i\in V$, 
{
i.e., the sum
of the $i$th row or column of $A$; or equivalently $k_i:=\#\mathcal{N_{\mathit{i}}}$ where $\mathcal{N_{\mathit{j}}}=\left\{ j\in V|\left(i,j\right)\in E\right\} $ is the set of all nearest neighbors of $i$.  
We will denote the minimal and maximum degree by $\delta$ and $\varDelta$,
respectively. Let $j$ be a node such that $j\in{\mathcal{N_{\mathit{i}}}}$. Then, we say that $j$ is a ``hub'', more correctly a ``local hub'', if it has the largest degree among all $j\in{\mathcal{N_{\mathit{i}}}}$.
We will denote by $K$ the diagonal matrix of vertex degrees.
\footnote{In the case of weighted graphs
the degree is often referred as \textit{strength}, but we will use the general term degree here in all cases.}

We use the following condensed notation across this paper. 
If $x_{\alpha}$ is a number depending on an index $\alpha$   --   in
the following typically $\alpha\in\left\{ -1,1\right\} $   --   then we
will write $x_{\alpha}$ to symbolize both $x_{-1}$ and $x_{1}$
depending on the choice on the index $\alpha$.
Let $\ell^2(V)$ be the finite-dimensional Hilbert space of functions on $V$ with respect to the inner product 
\[
\langle f,g\rangle=\sum_{v\in V}f(v)\overline{g(v)},\qquad f,g\in\ell^2(V).
\]
The standard graph Laplacian is an operator in $\ell^2(V)$ which is defined by
\begin{equation}
\bigl(\mathscr{L}f\bigr)(v)\coloneqq\sum_{w\in V:\,(v,w)\in E}{\varphi(vw) }\bigl(f(v)-f(w)\bigr),\qquad f\in\ell^2(V),\label{eq:path_Laplacian}
\end{equation}

where {$\varphi(v,w)\in W$ is the weight of the edge $(v,w)\in E$.} 

Finally, in the following ${1_n}$  will denote the all-ones column vector of order $n$;
$J_{n}$ the $n\times n$ all-one matrix; and $I_{n}$ the identity matrix of order $n$.

\section{Hubs-biased Laplacians and their spectra}\label{sec:spectrum}

Here we introduce the concepts of hubs-biased Laplacians in the context
of resistive networks.

\begin{defn}\label{def:conduct}
A \textit{conductance} function is a function $c:V\times V\rightarrow\mathbb{R}^{+}$
which {respects adjacency between vertices}, i.e., $\left(v,w\right)\in E$
if and only if $c\left(v,w\right)>0$.
\end{defn}

\begin{defn}
The \textit{total conductance} at a vertex $v$ is defined as

\begin{equation}\label{eq:conductance}
\begin{split}
c\left(v\right)&\coloneqq\sum_{\left(v,w\right)\in E}c\left(v,w\right)\\
&=\sum_{w\in V}c\left(v,w\right).
\end{split}
\end{equation}
\end{defn}

(The second equality in~\eqref{eq:conductance} holds in view of Definition~\ref{def:conduct}.)



{
To motivate the following definition, let us now consider a diffusive particle that does not necessarily hop to any nearest neighbor with the same probability. More precisely, in a \textit{hubs-attracting model} the hopping of a diffusive particle from a vertex $i$ to a target neighboring vertex $j$ is favored by a (comparatively) small degree of $i$  and a large degree of $j$ (``$j$ is a hub''), corresponding to a small ratio $\frac{k_i}{k_j}$. On the other hand, the hopping from $i$ to $j$ is disfavored by a (comparatively) large degree of $i$ and a (comparatively) small degree of $j$. Notice that these conditions for the hopping from $i$ to $j$ are different from the ones for the hopping from $j$ to $i$: hubs-attracting diffusion is \textit{not} symmetric.

Specularly, \textit{hubs-repelling} models are conceivable in which the hopping from $i$ to $j$ would be favored by a comparatively large degree of $i$ and a comparatively small degree of $j$, and disfavored by a comparatively small degree of $i$ and a comparatively large degree of $j$.

We will capture this intuition by the following.
}

\begin{defn}\label{def:laplacalpha}
 Let $G$ be a graph satisfying the Assumptions~\ref{assum:basic} and let $\alpha\in\left\{ -1,1\right\}.$
The \textit{hubs-biased Laplacian} corresponding to $\alpha$ is the operator
on $\ell^2(V)$ defined by 
\begin{equation}
\bigl(\mathscr{L}_{\alpha}f\bigr)(v)\coloneqq\sum_{w\in\mathcal N_v}c_{\alpha}\left(v,w\right)\bigl(f(v)-f(w)\bigr),\qquad f\in\ell^2(V),\label{eq:path_Laplacian-1}
\end{equation}
where here and in the following
\begin{equation}\label{eq:cavw}
c_{\alpha}\left(v,w\right):=\left(\dfrac{k_{v}}{k_{w}}\right)^{\alpha}.
\end{equation}

We call $\mathscr{L}_{\alpha}$ the \textit{hubs-repelling Laplacian} if $\alpha=1$ and hence $c_{\alpha}\left(v,w\right)=\dfrac{k_{v}}{k_{w}}$; or \textit{hubs-attracting Laplacian} if $\alpha=-1$ and hence $c_{\alpha}\left(v,w\right)=\dfrac{k_{w}}{k_{v}}$.
\end{defn}

\begin{rem}
Actually, we could easily extend this definition by allowing for any $\alpha\in {\mathbb R}$; for $\alpha=0$, corresponding to the unweighted
case, we would then recover the standard (discrete) Laplacian ${\mathscr{L}}_{0}$.
However, we will not explore this direction in this paper.
\end{rem}

\begin{rem}
In order to understand the main difference between the current approach and some previous approaches introduced for the analysis of weighted directed graphs we should point out the following. First, we analyze here unweighted and undirected graphs, which are then transformed into a very specific kind of weighted directed graphs. In \cite{Boley_2011,Young_2015,Bianchi_2019,Boyd_2021}, among others, the authors focus on weighted directed graphs. Then, they generate either asymmetric Laplacians $L^{a}$ \cite{Boley_2011,Bianchi_2019}
or symmetrized versions $L^{s}$ thereof \cite{Boley_2011,Boyd_2021}. Yet another approach was proposed by Young et al. \cite{Young_2015} where the resistance distances are calculated from a matrix $X=2Q^{T}{\Sigma}Q$ where $\tilde{\mathscr{L}}{\Sigma}+{\Sigma}\tilde{\mathscr{L}}^{T}=I_{n-1}$, $\tilde{\mathscr{L}}=QLQ^{T}$ and the matrix $Q$ obeys the following conditions: $Q1_n=0$, $QQ^{T}=I_{n-1}$ and $Q^{T}Q=I_n-(1/n)J_n$. The matrix $\tilde{\mathscr{L}}=Q{\mathscr{L}}Q^{T}$ is known as the reduced Laplacian. In a further work, Fitch \cite{Fitch_2019} demonstrated that the resistance distances obtained through the matrix $X$ for any pair of vertices in any connected, directed graph, are equal to the resistance distances obtained for a certain symmetric, undirected Laplacian on the same set of nodes and possibly admitting negative edge weights. As the hubs-biased Laplacian matrices are non-symmetric in general, it is obvious that these two approaches are not equivalent.

In the current study a different kind of matricial structure emerges, which is neither the asymmetric cases previously considered nor a symmetric one. The matrices $\mathscr{L}_{\alpha}$ correspond to the class of quasi-reciprocal matrices (see \cite{Harker_1987}), which has not been previously used in the analysis of graphs. An $n\times n$ matrix $M$ is called \textit{quasi-reciprocal} if
\begin{equation}
M_{ij}\neq0\Rightarrow M_{ji}=M_{ij}^{-1},\forall i,j=1,2,\ldots,n.
\end{equation}
Therefore, due to this notable difference the approach developed here for the resistance distance and related descriptions of undirected graphs are substantially different from the ones previously analyzed for weighted directed graphs \cite{Boley_2011,Young_2015,Bianchi_2019,Boyd_2021}.
\end{rem}

Let $e_{v}$, $v\in V$ be a standard orthonormal basis in $\ell^2(V)$
consisting of the vectors 
\begin{equation}
e_{v}(w)\coloneqq\begin{cases}
1 & \text{if \ensuremath{w=v}},\\[0.5ex]
0 & \text{otherwise}.
\end{cases}\label{defev}
\end{equation}

Then, {$\mathscr{L}_{\alpha}$} acts on the
vectors $e_{v}$ as follows: 
\begin{equation}
(\mathscr{L}_{\alpha}e_{v})(w)=\begin{cases}
c_{\alpha}(v) & \text{if \ensuremath{w=v}},\\[0.5ex]
-c_{\alpha}\left(v,w\right) & \text{if \ensuremath{(v,w)\in E}},\\[0.5ex]
0 & \text{otherwise},
\end{cases}\label{Lkev}
\end{equation}
where 
\begin{equation}\label{eq:cav}
c_{\alpha}(v)=\sum_{w\in\mathscr{\mathcal{N_{\mathit{v}}}}}c_{\alpha}\left(v,w\right)
\end{equation}
is the total $\alpha$-conductance of the vertex $v$. Then, the hubs-biased
matrices can be expressed as 
\begin{equation}\label{eq:lalpha-xi}
\mathscr{L}_{\alpha}=\Xi_{\alpha}-K^{\alpha}AK^{-\alpha},
\end{equation}
where 
\begin{equation}
\begin{split}
\Xi_{\alpha}&:=\diag\left(c_\alpha(v)\right)_{v\in V},\\
K&:=\diag\left(k_{v}\right)_{v\in V},
\end{split}
\label{eq:degree matrix}
\end{equation}
and $A$ is the unweighted adjacency matrix of the graph.

Let us note an elementary but important fact concerning the trace of (any) hubs-biased Laplacians.
\begin{lem}
\label{lem:Lemma_trace}
There holds
\begin{equation}
\trhb
\left(G\right)
:=\textnormal{tr}\left(\mathscr{L}_{1}\right)=\textnormal{tr}\left(\mathscr{L}_{-1}\right)=\sum_{v,w\in V}a_{vw}\dfrac{k_{v}}{k_{w}}.\label{eq:lemma_trace}
\end{equation}
\end{lem}

\begin{proof}
Let us, for $\alpha\in\{-1,1\}$, consider the
definition of total conductance and in view of \eqref{eq:conductance},

\begin{equation}
\textnormal{tr}\left(\mathscr{L}_{\alpha}\right)=\sum_{v\in V}c_{\alpha}\left(v\right)=\sum_{v,w\in V}a_{vw}\left(\dfrac{k_{v}}{k_{w}}\right)^{\alpha}.
\end{equation}
Now, because $a_{v,w}=1$ if and only if
$a_{w,v}=1$ (and in this case both addends $\dfrac{k_{v}}{k_{w}},\dfrac{k_{w}}{k_{v}}$
appear), we conclude that
\begin{equation}
\textnormal{tr}\left(\mathscr{L}_{\alpha}\right)=\sum_{v\in V}c_{\alpha}\left(v\right)=\sum_{v,w\in V}a_{vw}\dfrac{k_{v}}{k_{w}}:\label{eq:trace_2}
\end{equation}
since the right-hand side is independent of $\alpha$, the claim is proved.
\end{proof}

{
\begin{rem}
A rather natural generalization of Definition~\ref{def:laplacalpha} involves  infinite graphs. Let for a moment   --   unlike under the standing Assumptions~\ref{assum:basic}!   --   the graph $G$ be allowed to have infinitely many vertices. If the degree sequence $(k_v)_{v\in V}$ is still bounded, then by~\eqref{eq:cavw} and~\eqref{eq:cav} the sequence $(c_\alpha(v))_{v\in V}$ is bounded, too; it is then easy to see that the matrices $\Xi_\alpha$, $K^{\pm\alpha}$, and $A$ define bounded linear operators on the (infinite dimensional!) space $\ell^2(V)$. We conclude that for such \textit{infinite, uniformly locally bounded graphs} $\mathcal L_\alpha$ is well-defined as a bounded linear operator on $\ell^2(V)$, too. However, $A$ and hence $\mathcal L_\alpha$ have absolutely continuous spectrum, rather than a discrete set of eigenvalues as in the finite case: for this reason, this setting is not convenient for our purposes: for example, $\mathcal L_\alpha$ will not have finite trace.
\end{rem}
}

The double-sided bound
\begin{equation}
2m\dfrac{\delta}{\varDelta}\leq\trhb\left(G\right)\leq2m\dfrac{\varDelta}{\delta}\label{eq:inequalities_degree}
\end{equation}
on the hubs-biased trace immediately follows from~\eqref{eq:lemma_trace}, since $\delta\leq k_{v}\leq\varDelta$ for all $v\in V$,
and because  $\sum_{v,w\in V}a_{vw}=\sum_{v\in V}k_{v}=2m$ by the Handshaking
Lemma. (We recall that $\delta$ and $\varDelta$ denote the minimum and maximum degree of the
vertices of $G$, respectively, and $m$ is the number of edges of $G$.)

The estimates in~\eqref{eq:inequalities_degree} are rough, yet sharp as both inequalities become equalities for complete graphs. (We should observe that if $G$ is complete, then
$\trhb\left(G\right)$ also agrees with the trace of the standard discrete Laplacian.)

We are able to provide a few improved estimates.

{
\begin{lem}
There holds
\begin{equation}
\dfrac{1}{\varDelta}\sum_{v\in V}k_{v}^{2}\leq\trhb\left(G\right)\leq\dfrac{1}{\delta}\sum_{v\in V}k_{v}^{2}.\label{eq:bound_square}
\end{equation}
\end{lem}
\begin{proof}
The bounds follow directly from~\eqref{eq:lemma_trace}, since
\begin{equation}\label{eq:squared_degree}
\sum_{w\in V}a_{vw}k_{v}=k_{v}\sum_{w\in V}a_{vw}=k_{v}^{2}\qquad\hbox{for all $v\in V$}.
\end{equation}
Summing over $v\in V$ yields the claimed bounds.
\end{proof}
}

Because of the well-known identity
\[
\sum_{v\in V}k_{v}^{2}=\sum_{(v,w)\in E}\left(k_{v}+k_{w}\right),
\]
the estimates in~\eqref{eq:bound_square} imply those in \eqref{eq:inequalities_degree},
but are clearly sharper if, e.g., $G$ is bi-regular (recall that $G$ is \textit{bi-regular} if $V$ can be partitioned in $V_{1},V_{2}$ with $k_v\equiv k_1\in \N$ for all $v\in V_1$; and $k_w\equiv k_2\in \N$ for all $w\in V_2$).

{
More explicit bounds in \eqref{eq:bound_square} can be obtained using
known estimates on $\sum_{v\in V}k_{v}^{2}$ and $\sum_{v\in V}k_{v}^{-1}$, the so-called (first) \textit{Zagreb index} and \textit{Randić index} of $G$, respectively: we refer to~\cite{AliGutMil18} for a comprehensive survey of results on this topic, including improving bounds for special classes, like planar or triangle-free graphs, that might be of interest in applications. In particular, we mention the following.
}

\begin{cor}
There holds
\begin{equation}
\dfrac{4m^{2}}{n\varDelta}\leq\trhb\left(G\right)\leq\dfrac{2m\left(2m+\left(n-1\right)\left(\varDelta-\delta\right)\right)}{\delta\left(n+\varDelta-\delta\right)}.\label{eq:estimate_trace}
\end{equation}
{
The lower estimate becomes an equality if and only if $G$ is regular. The upper estimate becomes an equality if and only if $G$ is a graph with $t$ vertices of degree $n-1$ and the remaining $n-t$ vertices forming an independent set.}
\end{cor}
\begin{proof}
The claimed bounds follow from the known estimates
\begin{equation}\label{eq:dassur}
\dfrac{4m^{2}}{n}\leq\sum_{v\in V}k_{v}^{2}\leq\dfrac{2m\left(2m+\left(n-1\right)\left(\varDelta-\delta\right)\right)}{n+\varDelta-\delta}:
\end{equation}
cf.~\cite[Remark 4 and Remark 5]{Cio06} and~\cite{Das04}, where also the extremal graphs are characterized.
\end{proof}

\begin{rem}
(1) The lower estimate in \eqref{eq:estimate_trace} is significantly
better than the lower estimate in \eqref{eq:inequalities_degree}:
indeed, the net effect is like replacing $\delta$ by the average
degree $\dfrac{2m}{n}$ in \eqref{eq:inequalities_degree}. The upper
estimate in \eqref{eq:estimate_trace} is better than the upper estimate
in \eqref{eq:inequalities_degree} if and only if 

\begin{equation}
2m+\left(n-1\right)\left(\varDelta-\delta\right)\leq\varDelta\left(n+\varDelta-\delta\right),\label{eq:casual}
\end{equation}
which is sometimes the case, for instance for any regular graph, and
sometimes not, for instance for any path on more than 2 edges. The
equality in \eqref{eq:casual} holds for complete graphs.

(2)
Further estimates on the hubs-biased trace may be obtained re-writing~\eqref{eq:lemma_trace} in alternative ways, including
\begin{equation}
\sum_{v,w\in V}a_{vw}\dfrac{k_{v}}{k_{w}}=\sum_{v\in V}k_{v}\left(\sum_{w\in V}\dfrac{a_{vw}}{k_{w}}\right)=\sum_{v\in V}k_{v}\left(\sum_{w\in\mathcal{N}_{v}}\dfrac{1}{k_{w}}\right),\label{eq:general_trace_1}
\end{equation}
or
\begin{equation}
\sum_{v,w\in V}a_{vw}\dfrac{k_{v}}{k_{w}}=\sum_{v\in V}\dfrac{1}{k_{v}}\left(\sum_{w\in V}a_{vw}k_{w}\right)=\sum_{v\in V}\dfrac{1}{k_{v}}\left(\sum_{w\in\mathcal{N}_{v}}k_{w}\right)\label{eq:general_trace_2}
\end{equation}
{(we recall that $\mathcal N_v$ is the set of nearest neighbors of $v$).}
Considering the minima $\delta_{\mathcal{N}_{v}}$ and maxima $\varDelta_{\mathcal{N}_{v}}$ of the degree function in neighborhoods $\mathcal N_v$  we can deduce from \eqref{eq:general_trace_1}
and \eqref{eq:general_trace_2} the following sharper estimates:

\begin{equation}
\sum_{v\in V}\dfrac{k_{v}^{2}}{\varDelta_{\mathcal{N}_{v}}}\leq\trhb\left(G\right)\leq\sum_{v\in V}\dfrac{k_{v}^{2}}{\delta_{\mathcal{N}_{v}}},\label{eq:lower_estimate}
\end{equation}

\begin{equation}
\sum_{v\in V}\delta_{\mathcal{N}_{v}}\leq\trhb\left(G\right)\leq\sum_{v\in V}\varDelta_{\mathcal{N}_{v}},
\end{equation}
which shows, for instance, that $\trhb\left(G\right)=2pq$
for the complete bipartite graph $K_{p,q}.$

(3) Invoking Titu's Lemma we also obtain
\begin{equation}
\sum_{w\in\mathcal{N}_{v}}\dfrac{1}{k_{w}}\geq\dfrac{\left(\sum_{w\in\mathcal{N}_{v}}1\right)^{2}}{\sum_{w\in\mathcal{N}_{v}}k_{w}}=\dfrac{k_{v}^{2}}{\sum_{w\in\mathcal{N}_{v}}k_{w}}.
\end{equation}

Because Titu's Lemma is equivalent to the Cauchy--Schwarz inequality,
the previous inequality becomes an equality if and only if  $\left(1\right)_{w\in\mathcal{N}_{v}}$
and $\left(k_{w}\right)_{w\in\mathcal{N}_{v}}$ are linearly dependent,
i.e., if and only if  the degree function is constant on each neighborhood, which
is the case for instance in regular and bi-regular graphs. We finally deduce the estimate
\begin{equation}
\trhb\left(G\right)\geq\sum_{v\in V}\dfrac{k_{v}^{3}}{\sum_{w\in\mathcal{N}_{v}}k_{w}},
\end{equation}
which of course implies the lower estimate in \eqref{eq:lower_estimate}.
\end{rem}

Let us now collect a few important properties of the hubs-biased Laplacian matrices of any graphs satisfying our standing Assumptions~\ref{assum:basic}.

\begin{thm}
\label{thm:real spectrum}
Let $\alpha\in\{-1,1\}$. Then the hubs-biased Laplacian matrix ${\mathscr{L}}_{\alpha}$
enjoys the following properties:

(i) its eigenvalues are real;

(ii) it is positive semidefinite;

(iii) $\textnormal{{rank\,}}{\mathscr{L}}_{\alpha}=n-\mathscr{C}$;

(iv) it can be diagonalized as ${\mathscr{L}}_{\alpha}=\left(KU_{\alpha}\right)\Lambda_{\alpha}\left(KU_{\alpha}\right)^{-1}$, where
$\Xi_{\alpha}-A=U_{\alpha}\Lambda_{\alpha}U_{\alpha}^{-1}$ and
$K$ is as in \eqref{eq:degree matrix}.
\end{thm}

\begin{proof}
{ 
(i) Under the Assumptions~\ref{assum:basic}, no vertex is isolated, hence $k_i\ge 1$ for all $i$. Therefore, for any $\alpha\in\left\{ -1,1\right\}$ the diagonal matrix $K^\alpha$ is invertible and we observe that}
\begin{equation}
\begin{split}{\mathscr{L}}_{\alpha} & =\Xi_{\alpha}-K^{\alpha}AK^{-\alpha}\\
 & =K^{\alpha}\left(K^{-\alpha}\Xi_{\alpha}K^{\alpha}-A\right)K^{-\alpha}\\
 & =K^{\alpha}\left(\Xi_{\alpha}-A\right)K^{-\alpha},
\end{split}
\end{equation}
{since $\Xi_\alpha$ is diagonal, too. We conclude that}
 {${\mathscr{L}}_{\alpha}$}
is similar to the symmetric matrix $\left(\Xi_{\alpha}-A\right)$,
and so their eigenvalues are real.

(ii) Now let $x\in\ell^2(V)$ and $x\neq{0}$. Then,
we can write

\begin{equation}
\begin{split}x^{T}\left(\Xi_{\alpha}-A\right)x & =\sum_{\left(i,j\right)\in E}\left(\left(k_{i}^{\alpha/2}k_{j}^{-\alpha/2}\right)x_{i}-\left(k_{j}^{\alpha/2}k_{i}^{-\alpha/2}\right)x_{j}\right)^{2}\end{split}\ge 0.
\end{equation}
Therefore, because {$\Xi_{\alpha}-A$}
and ${\mathscr{L}}_{\alpha}$ are similar, we have that ${x}^{T}{\mathscr{L}}_{\alpha}{x}\geq0$.

(iii)
Let us now prove that the dimension of the null space of {${\mathscr{L}}_{\alpha}$}
is $\mathscr{C}$. Let ${z}$ be a vector such that ${\mathscr{L}}_{\alpha}{z}=0$.
This implies that for every $\left(i,j\right)\in E$, $z_{i}=z_{j}$.
Therefore ${z}$ takes the same value on all vertices of the same
connected component, which indicates that the dimension of the null
space is $\mathscr{C}$, and so $\textnormal{{rank\,}}{\mathscr{L}}_{\alpha}=n-\mathscr{C}$.

(iv) Finally, we also have that {because $\Xi_{\alpha}-A$
is symmetric we can write it as: $\Xi_{\alpha}-A=U_{\alpha}\Lambda_{\alpha}U_{\alpha}^{-1}$.
Thus, ${\mathscr{L}}_{\alpha}=\left(KU_{\alpha}\right)\Lambda_{\alpha}\left(KU_{\alpha}\right)^{-1}$
which indicates that all hubs-biased Laplacians are diagonalizable.}
\end{proof}
We know from Theorem \ref{thm:real spectrum} that all eigenvalues of any hubs-biased Laplacian are real positive: we will denote them by $\rho_{\alpha,1}\ldots,\rho_{\alpha_n}$ in ascending order, i.e.,
\[
0=\rho_{\alpha,1}\le \ldots \le \rho_{\alpha,n}.
\] 
Let us conclude this section deducing some estimates on them. {To avoid trivialities, and in view of Theorem~\ref{thm:real spectrum}.(iii), in the remainder of this section we always assume the graph $G$ to be connected.}

\begin{cor}\label{cor:spectral_bound}
Let $\alpha\in\left\{ -1,1\right\}$. Then we have
\begin{equation}
\rho_{\alpha,n}\geq\frac{4m^2}{n(n-1)\varDelta}
\end{equation}
and
\begin{equation}\label{eq:cor-upp}
\left.\begin{array}{c}
\dfrac{2m\left(2m+\left(n-1\right)\left(\varDelta-\delta\right)\right)}{\delta(n-1)\left(n+\varDelta-\delta\right)}\\[15pt]
\dfrac{2m\varDelta}{(n-1)\delta}
\end{array}\right\} \geq\rho_{\alpha,2}.
\end{equation}
\end{cor}

We note in passing that on a $\left(k+1\right)$-star, the upper bounds in~\eqref{eq:cor-upp} reduce to
\begin{equation}
\left.\begin{array}{c}
2+k\\
2k+1
\end{array}\right\} \geq\rho_{\alpha,2,}
\end{equation}
which does not prevent {$\rho_{\alpha,2}$}
to tend to $+\infty$ as the graph grows.

We also remark that these estimates are only really interesting for
non-regular graphs, because in regular graphs the hubs-biased Laplacians
coincide with the standard discrete Laplacian for which many different
bounds are known for its eigenvalues.
\begin{proof}
Observe that $\dfrac{1}{n-1}\trhb\left(G\right)$ yields
the arithmetic mean of all non-zero eigenvalues of ${\mathscr{L}}_{\alpha}$;
in particular, the lowest non-zero eigenvalue $\rho_{\alpha,2}$ cannot
be larger than $\dfrac{1}{n-1}\trhb\left(G\right)$
while the largest eigenvalue $\rho_{\alpha,n}$ cannot be larger that
$\dfrac{1}{n-1}\trhb\left(G\right)$. Taking into account
\eqref{eq:inequalities_degree} and \eqref{eq:estimate_trace} we
deduce the claimed estimate.
\end{proof}

\begin{rem}\label{rem:gers}
(1) The naive upper bound
\begin{equation}\label{eq:naiveup}
\rho_{\alpha,n}\leq2\max_{v\in V}\left({\mathscr{L}}_{\alpha}\right)_{vv}
\end{equation}
on the largest eigenvalue of ${\mathscr{L}}_{\alpha}$ follows
from Ger\v{s}gorin's Theorem, since by \eqref{Lkev} $\left({\mathscr{L}}_{\alpha}\right)_{vv}=\sum_{w\in V}\left|\left({\mathscr{L}}_{\alpha}\right)_{vw}\right|$.
Now,
\begin{equation}
\begin{split}
\left({\mathscr{L}}_{-1}\right)_{vv}&=\sum_{w\in\mathcal{N}_{v}}\dfrac{k_{w}}{k_{v}}\leq\dfrac{\varDelta_{\mathcal{N}_{v}}}{k_{v}}\sum_{w\in\mathcal{N}_{v}}1=\varDelta_{\mathcal{N}_{v}}\le \varDelta,\\
\left({\mathscr{L}}_{1}\right)_{vv}&=\sum_{w\in\mathcal{N}_{v}}\dfrac{k_{v}}{k_{w}}\leq\dfrac{k_{v}}{\delta}
\sum_{w\in\mathcal{N}_{v}}1=\dfrac{k_{v}^{2}}{\delta}\leq\dfrac{\varDelta^{2}}{\delta},
\end{split}
\end{equation}
which are both sharp for regular graphs. Plugging these expressions in~\eqref{eq:naiveup} yields
\begin{equation}
\rho_{-1,n}\leq 2\varDelta
\end{equation}
and
\begin{equation}
\rho_{1,n}\leq2\dfrac{\varDelta^{2}}{\delta}.
\end{equation}

(2) In order to show some lower bounds on $\rho_{\alpha,2}$, we remark that the standard discrete
Laplacian $\mathcal{\mathscr{L}}$ and ${\mathscr{L}}_{\alpha}$
share the null space, cf.\ the proof of Theorem~\ref{thm:real spectrum}.(iii). Then, we can quotient it out, {thus studying the lowest eigenvalue of $\mathcal{\mathscr{L}}$ and ${\mathscr{L}}_{\alpha}$
on $\mathbb{C}^{n}/\left\langle {\mathbf 1}\right\rangle$, the space of vectors orthogonal to the vector ${\mathbf 1}=(1,\ldots,1)$}. Taking a normalized vector $f$ in this space one sees that
\begin{equation}
\dfrac{\delta}{\varDelta}\left(\mathscr{L}f,f\right)\leq\left(\mathscr{L}_{\alpha}f,f\right).
\end{equation}

In particular, choosing $f$ to be an eigenfunction associated with
the lowest non-zero eigenvalue $\rho_{\alpha,2}$ of $\mathscr{L}_{\alpha}$
and applying the {Courant--Fischer min-max} characterization of the eigenvalues of the Hermitian matrix we deduce
\begin{equation}
\dfrac{\delta}{\varDelta}\rho_{2}\leq\rho_{\alpha,2},
\end{equation}
where $\rho_{2}$ is the second lowest eigenvalue of the discrete
Laplacian, i.e., the algebraic connectivity \cite{Algebraic connectivity}.
Alternatively, for $\alpha=-1$ we can use
\begin{equation}
\delta\left(\mathscr{L}_{\textnormal{norm}}f,f\right)\leq\left(\mathscr{L}_{\alpha}f,f\right),
\end{equation}
and deduce

\begin{equation}
\delta\rho_{2,\textnormal{norm}}\leq\rho_{\alpha,2},
\end{equation}
where $\mathscr{L}_{\textnormal{norm}}\coloneqq K^{-1/2}\mathscr{L}K^{-1/2}$
is the normalized Laplacian. Now we can either apply explicit formulae
for $\rho_{2}$ and { $\rho_{2,\textnormal{norm}}$
for classes of graphs or use general estimates like}

\begin{equation}
\rho_{2}\geq2\eta\left(1-\cos\dfrac{\pi}{n}\right),
\end{equation}
from \cite{Fiedler} or

\begin{equation}
\rho_{2,\textnormal{norm}}\geq\left\{ \begin{array}{c}
\dfrac{1}{Dn},\\
1-\cos\dfrac{\pi}{m},
\end{array}\right.
\end{equation}
from \cite{Chung,spectral gap}, respectively, where $\eta\geq1$
is the edge connectivity and $D$ is the diameter of $G$.

{
(3) Further sharp bounds on the Zagreb index are known, including~\cite[Theorem 2.3 and Theorem 2.6]{DasZagreb}, immediately yielding
\[
\varDelta+\frac{(2m-\varDelta)^2}{\varDelta(n-1)}+\frac{2(n-2)(\varDelta_2-\delta)^2}{\varDelta(n-1)^2}\le 
\trhb(G) \le 
\frac{(n+1)m-\varDelta(n-\varDelta)}{\delta}+\frac{2(m-\varDelta)^2}{\delta(n-2)},
\]
where $\varDelta_2$ denotes the second largest degree of $G$, the upper bound holding under the additional assumption that $n\ge 3$. (A characterization of the extremal graphs where equality is attained is available, too, but is rather technical.)

Following the same proof as in Corollary~\ref{cor:spectral_bound} finally yields the bounds
\[
\rho_{\alpha,n}\ge \frac{\varDelta}{n-1}+\frac{(2m-\varDelta)^2}{\varDelta(n-1)^2}+\frac{2(n-2)(\varDelta_2-\delta)^2}{\varDelta(n-1)^3}
\]
and
\[
\frac{(n+1)m-\varDelta(n-\varDelta)}{\delta(n-1)}+\frac{2(m-\varDelta)^2}{\delta(n-2)(n-1)}\ge \rho_{\alpha,2}.
\]
}
\end{rem}

\section{Hubs-biased Resistance Distance}

We adapt here some general definitions of resistive networks to the
case of hubs-biased systems, mainly following the classic formulations
given in \cite{resistance_0,resistance_1,resistance_5}. 
Let us
consider $G$ as a resistive network in which every edge $\left(v,w\right)\in E$
has edge resistance $r_{\alpha}\left(v,w\right)\coloneqq c_{\alpha}^{-1}\left(v,w\right)$.

Let us consider the connection of a voltage between the vertices $v$
and $w$, and let $i\left(v,w\right)>0$ be the net current out the
source $v$ and into the sink $w$, such that $i\left(v,w\right)=-i\left(w,v\right).$
Then, according to the first Kirchhoff's law we have that $\sum_{w\in\mathcal{N}_{v}}i\left(v,w\right)=I$
if $v$ is a source, $\sum_{w\in\mathcal{N}_{v}}i\left(v,w\right)=-I$
if $v$ is a sink, or zero otherwise, {where we have denoted by $I$ the net current flowing through the whole network}. The application of the second
Kirchhoff's law, namely that $\sum_{\left(v,w\right)\in C}i\left(v,w\right)r_{\alpha}\left(v,w\right)=0$
where $C$ is a cycle with edges labeled in consecutive order, implies
that a potential $\mathscr{V}$ may be associated with any vertex
$v$, such that for all edges

\begin{equation}
i\left(v,w\right)r_{\alpha}\left(v,w\right)=\mathscr{V}\left(v\right)-\mathscr{V}\left(w\right),
\end{equation}
which represents the Ohm's law, and where $i$ and $\mathscr{V}$
depend on the net current $I$ and on the pair of vertices where the
voltage source has been placed. Let us now define formally the {\textit{hubs-biased effective resistance}}, which is the resistance
of the total system when a voltage source is connected across a corresponding
pair of vertices. Throughout this section we are still adopting the notation in Section~\ref{sec:spectrum}: in particular, $\alpha\in \{-1,1\}$.
\begin{defn}
The \textit{hubs-biased effective resistance} between the vertices $v$
and $w$ of $G$ is

\begin{equation}
\Omega_{\alpha}\left(v,w\right):=\dfrac{\mathscr{V}\left(v\right)-\mathscr{V}\left(w\right)}{I}.
\end{equation}
\end{defn}

We now prove the following result, showing that hubs-biased resistance between any two vertices of $G$ is a squared Euclidean distance.

\begin{lem}
 \label{lem:pseudoinverse}
Let $v,w\in V$. Then for $\alpha\in\{-1,1\}$ the hubs-biased resistance $\Omega_{\alpha}\left(v,w\right)$
is given by
\begin{equation}
\Omega_{\alpha}\left(v,w\right)=\mathscr{L}_{\alpha}^{+}\left(v,v\right)+\mathscr{L}_{\alpha}^{+}\left(w,w\right)-\mathscr{L}_{\alpha}^{+}\left(v,w\right)-\mathscr{L}_{\alpha}^{+}\left(w,v\right),\label{eq:resistance}
\end{equation}
where $\mathscr{L}_{\alpha}^{+}$ stands for the Moore--Penrose pseudoinverse
of $\mathscr{L}_{\alpha}$.
\end{lem}

%
\begin{proof}
First, we will prove that 
\begin{equation}
\Omega_{\alpha}\left(v,w\right)=\left({e}_{v}-{e}_{w}\right)^{T}\mathscr{L}_{\alpha}^{+}\left(v,v\right)\left({e}_{v}-{e}_{w}\right),
\end{equation}
where ${e}_{v}$ is the vector with all entries equal to zero
except the one corresponding to vertex $v$ which is equal to one.
Using {the second Kirchhoff's law we have}

\begin{equation}
\sum_{w\in\mathcal{N}_{v}}\dfrac{1}{r\left(v,w\right)}\left(\mathscr{V}\left(v\right)-\mathscr{V}\left(w\right)\right)=
\begin{cases}
I \qquad&\hbox{if $v$ is a source},\\
-I \qquad&\hbox{if $v$ is a sink},\\
0 & \hbox{otherwise},
\end{cases}
\end{equation}
which can also be written as
\begin{equation}
c_{\alpha}\left(v\right)\mathscr{V}\left(v\right)-\sum_{w=1}^{n}\dfrac{1}{r\left(v,w\right)}\mathscr{V}\left(w\right)=\begin{cases}
I \qquad&\hbox{if $v$ is a source},\\
-I \qquad&\hbox{if $v$ is a sink},\\
0 & \hbox{otherwise}.
\end{cases}
\end{equation}

Let us write it in matrix-vector form as
\begin{equation}
\mathscr{\mathscr{L}_{\alpha}{V}}=I\left({e}_{v}-{e}_{w}\right).\label{eq:Laplcian equation}
\end{equation}

Due to the fact that the right-hand side of \eqref{eq:Laplcian equation} is orthogonal
to ${1}$ we can obtain ${\mathscr{V}}$ as

\begin{equation}
\mathscr{V}\left(v\right)-\mathscr{V}\left(w\right)=\left({e}_{v}-{e}_{w}\right)^{T}\mathscr{{V}}=I\left({e}_{v}-{e}_{w}\right)^{T}\mathscr{L}_{\alpha}^{+}\left({e}_{v}-{e}_{w}\right).
\end{equation}

Then, using the definition of the effective resistance we have
\begin{equation}
\Omega_{\alpha}\left(v,w\right)=\dfrac{\mathscr{V}\left(v\right)-\mathscr{V}\left(w\right)}{I}=\left({e}_{v}-{e}_{w}\right)^{T}\mathscr{L}_{\alpha}^{+}\left({e}_{v}-{e}_{w}\right).
\end{equation}
Now, because $\left({e}_{v}-{e}_{w}\right)^{T}\mathscr{L}_{\alpha}^{+}\left({e}_{v}-{e}_{w}\right)=\mathscr{L}_{\alpha}^{+}\left(v,v\right)+\mathscr{L}_{\alpha}^{+}\left(w,w\right)-\mathscr{L}_{\alpha}^{+}\left(v,w\right)-\mathscr{L}_{\alpha}^{+}\left(w,v\right)$
we only remain to prove that it is a distance. {Let
$\mathscr{L}_{\alpha}=V_{\alpha}\Lambda_{\alpha}V_{\alpha}^{-1}$,
where }$V_{\alpha}=KU_{\alpha}.$
Then
$\mathscr{L}_{\alpha}^{+}=V_{\alpha}\Lambda_{\alpha}^{+}V_{\alpha}^{-1},$
with $\Lambda_{\alpha}^{+}$ being the
Moore--Penrose pseudoinverse of the diagonal matrix of eigenvalues
of$\mathscr{L}_{\alpha},$ i.e., the diagonal matrix whose $i$th
entry is
\[
\Lambda_{\alpha}^{+}\left(i,i\right)=
\begin{cases}
0\quad &\hbox{if the $i$th eigenvalue is $0$},\\
\rho_{\alpha,i}^{-1}\quad &\hbox{if the $i$th eigenvalue is $\ne 0$}.
\end{cases}
\]

Let us write the right-hand side of \eqref{eq:resistance} as
\begin{equation}
{v}_{\alpha}\Lambda_{\alpha}^{+}{u}_{\alpha}^{T}+{w}_{\alpha}\Lambda_{\alpha}^{+}{w}_{\alpha}^{T}-{v}_{\alpha}\Lambda_{\alpha}^{+}{w}_{\alpha}^{T}-{w}_{\alpha}\Lambda_{\alpha}^{+}{v}_{\alpha}^{T},
\end{equation}
where ${v}$ and ${w}$ are the corresponding rows of $V_{\alpha}$
for the vertices $u$ and $w$, respectively. Then, we have

\begin{equation}
\begin{split} & \mathscr{L}_{\alpha}^{+}\left(u,u\right)+\mathscr{L}_{\alpha}^{+}\left(w,w\right)-\mathscr{L}_{\alpha}^{+}\left(u,w\right)-\mathscr{L}_{\alpha}^{+}\left(w,u\right)\\
 & ={v}_{\alpha}\left(\Lambda_{\alpha}^{+}{v}_{\alpha}^{T}-\Lambda_{\alpha}^{+}{w}_{\alpha}^{T}\right)-{w}_{\alpha}\left(\Lambda_{\alpha}^{+}{v}_{\alpha}^{T}-\Lambda_{\alpha}^{+}{w}_{\alpha}^{T}\right),\\
 & =\left({v}_{\alpha}-{w}_{\alpha}\right)\left(\Lambda_{\alpha}^{+}{v}_{\alpha}^{T}-\Lambda_{\alpha}^{+}{w}_{\alpha}^{T}\right)\\
 & =\left({v}_{\alpha}-{w}_{\alpha}\right)\Lambda_{\alpha}^{+}\left({v}_{\alpha}-{w}_{\alpha}\right)^{T}\\
 & =\left(\left({v}_{\alpha}-{w}_{\alpha}\right)\sqrt{\Lambda_{\alpha}^{+}}\right)\left(\left({v}_{\alpha}-{w}_{\alpha}\right)\sqrt{\Lambda_{\alpha}^{+}}\right)^{T}\\
 & =\left(\mathscr{V}_{\alpha}\left(v\right)-\mathscr{V}_{\alpha}\left(w\right)\right)^{T}\left(\mathscr{V}_{\alpha}\left(v\right)-\mathscr{V}_{\alpha}\left(w\right)\right)\\
 & =\left\Vert \mathscr{V}_{\alpha}\left(v\right)-\mathscr{V}_{\alpha}\left(w\right)\right\Vert ^{2},
\end{split}
\end{equation}
where $\mathscr{V}_{\alpha}\left(v\right)={v}_{\alpha}\sqrt{\Lambda_{\alpha}^{+}}$
is the position vector of the vertex $v$ in the Euclidean space induced
by the hubs-repelling Laplacian.
\end{proof}

\begin{cor}\label{cor:corollary_omega}
Let $\alpha\in\{-1,1\}$ and  $\mathscr{L}_{\alpha}=V_{\alpha}\Lambda_{\alpha}V_{\alpha}^{-1}$,
where 
 $V_{\alpha}:=KU_{\alpha}$ for $V_{\alpha}:=\left[{\psi}_{\alpha,1},{\psi}_{\alpha,2},\cdots,{\psi}_{\alpha n}\right]$
and $\Lambda_{\alpha}:=\diag \left(\rho_{\alpha,k}\right)$.
Then,
\begin{equation}
\Omega_{\alpha}\left(u,w\right)=\sum_{k=2}^{n}\rho_{\alpha,k}^{-1}\left(\psi_{\alpha,k,u}-\psi_{\alpha,k,w}\right)^{2}.
\end{equation}

\end{cor}

\begin{proof}
It is easy to see from Lemma~\ref{lem:pseudoinverse} that

\begin{equation}
\begin{split} 
\Omega_{\alpha}\left(u,w\right)
 & =\sum_{k=2}^{n}\rho_{\alpha,k}^{-1}\psi_{\alpha,k,u}^{2}+\sum_{k=2}^{n}\rho_{\alpha,k}^{-1}\psi_{\alpha,k,w}^{2}-2\sum_{k=2}^{n}\rho_{\alpha,k}^{-1}\psi_{\alpha,k,u}\psi_{\alpha,k,w}\\
 & =\sum_{k=2}^{n}\rho_{\alpha,k}^{-1}\left(\psi_{\alpha,k,u}^{2}+\psi_{\alpha,k,w}^{2}-2\psi_{\alpha,k,u}\psi_{\left\{ R,A\right\} ,k,w}\right)\\
 & =\sum_{k=2}^{n}\rho_{\alpha,k}^{-1}\left(\psi_{\alpha,k,u}-\psi_{\alpha,k,w}\right)^{2}.
\end{split}
\end{equation}
\end{proof}

\begin{cor}
Let $\alpha=\{-1,1\}$ and $0=\rho_{\alpha,1}<\rho_{\alpha,2}\leq\cdots\leq\rho_{\alpha,n}$
be the eigenvalues of $\mathscr{L}_{\alpha}$. Then,

\begin{equation}
\dfrac{2}{\rho_{\alpha,n}}\leq\Omega_{\alpha}\left(u,w\right)\leq\dfrac{2}{\rho_{\alpha,2}}.
\end{equation}
\end{cor}

\begin{proof}
\textcolor{black}{Using Corollary \ref{cor:corollary_omega} and the
fact that $\rho_{\alpha,2}$ is the smallest eigenvalue of $\mathscr{L}_{\alpha}$
we have the upper bound if all the eigenvalues are equal to $\rho_{\alpha,2}$.
The result follows from the fact that }$\sum_{k=1}^{n}\psi_{\alpha,k,u}^{2}=1$
and $\sum_{k=1}^{n}\psi_{\alpha,k,u}\psi_{\alpha,k,w}=0$ for every
$u\neq w$. \textcolor{black}{The lower bound is obtained similarly
by the fact that $\rho_{\alpha,n}$ is the largest eigenvalue of $\mathscr{L}_{\alpha}$.}
\end{proof}

\subsection{Hubs-biased Kirchhoff index}

In full analogy with the definition of the so-called Kirchhoff index
by Klein and Randi\'{c} \cite{resistance_1} we define here the hubs-biased
Kirchhoff indices of a graph.
\begin{defn}
Let $\alpha\in\{-1,1\}$. The total hubs-biased resistance distance, or \textit{hubs-biased Kirchhoff index}, of a graph $G$ is defined as

\begin{equation}
\mathcal{\mathscr{R}}_{\alpha}\left(G\right)=\sum_{v<w}\Omega_{\alpha}\left(v,w\right).
\end{equation}
\end{defn}

\begin{lem}\label{lem:Lemma_Kirchhoff}
Let $\alpha\in\{-1,1\}$ and $0=\rho_{\alpha,1}<\rho_{\alpha,2}\leq\cdots\leq\rho_{\alpha,n}$
be the eigenvalues of $\mathscr{L}_{\alpha}$. Then,
\begin{equation}
\mathcal{\mathscr{R}}_{\alpha}\left(G\right)=n\sum_{k=2}^{n}\dfrac{1}{\rho_{\alpha,k}}.
\end{equation}
\end{lem}

\begin{proof}
Let us write the sum of the hubs-biased resistance
distances as
\begin{equation*}
\begin{split}
\dfrac{1}{2}\sum_{v,w\in V}\Omega_{\alpha}\left(u,w\right) & =\dfrac{1}{2}\left({1}^{T}\textnormal{diag\ensuremath{\left(\mathscr{L}_{\alpha\ }^{+}\right){1}^{T}}}{1}+{1}^{T}{1}\left(\diag\left(\mathscr{L}_{\alpha}^{+}\right)\right)^{T}{1}-{1}^{T}\mathscr{L}_{\alpha}^{+}{1}-{1}^{T}\left(\mathscr{L}_{\alpha}^{+}\right)^{T}{1}\right)\\
 & =\dfrac{1}{2}\left(2n\textnormal{tr}\left(\mathscr{L}_{\alpha}^{+}\right)\right)\\
 & =n\sum_{k=2}^{n}\dfrac{1}{\rho_{\alpha,k}},
\end{split}
\end{equation*}
where $\textnormal{tr}\left(\mathscr{L}_{\alpha}^{+}\right)$ is
the trace of $\mathscr{L}_{\alpha}^{+}$.
\end{proof}
\begin{cor}\label{cor:bound RR}
Let $\alpha\in\{-1,1\}$ and $0=\rho_{\alpha,1}<\rho_{\alpha,2}\leq\cdots\leq\rho_{\alpha,n}$
be the eigenvalues of $\mathscr{L}_{\alpha}$ for $G$ with $n\geq2$.
Then,
\begin{equation}
\dfrac{n\left(n-1\right)}{\rho_{\alpha,n}}\mathcal{\leq\mathcal{\mathscr{R}}_{\mathnormal{\alpha}}}\left(G\right)\leq\dfrac{n\left(n-1\right)}{\rho_{\alpha,2}}.
\end{equation}
\end{cor}

\medskip{}

\begin{lem}
Let $\alpha\in\{-1,1\}$. Then
\begin{equation}
\mathcal{\mathcal{\mathscr{R}}_{\mathnormal{\alpha}}}\left(G\right)\geq n-1,
\end{equation}
with equality if and only if $G=K_{n}$.
\end{lem}

\begin{proof}
Let $\mathcal{S}$ be the set of all column vectors $x$ such that
$x\cdot x=1$, $x\cdot e$= 0. Then, it is known that 

\begin{equation}
\rho_{\alpha,2}=\min_{x\in S}x^{T}\mathscr{L}_{\alpha}x.
\end{equation}
Then, we can show that the matrix $\mathscr{\tilde{L}}_{\alpha}\coloneqq\mathscr{L}_{\alpha}-\rho_{\alpha,2}\left(I_{n}-n^{-1}J_{n}\right)$
is also positive semidefinite. Since $\mathscr{\tilde{L}}_{\alpha}e=0$
we have

\begin{equation}
y^{T}\mathscr{\tilde{L}}_{\alpha}y=c_{2}^{2}x^{T}\mathscr{\tilde{L}}_{\alpha}x=c_{2}^{2}\left(x^{T}\mathscr{L}_{\alpha}x-\rho_{\alpha,2}\right)\geq0.
\end{equation}
Thus

\begin{equation}
\min_{v}\mathscr{L}_{\alpha}\left(v,v\right)-\rho_{\alpha,2}\left(1-n^{-1}\right)\geq0.
\end{equation}
We can then write, 

\begin{equation}
\rho_{\alpha,2}\leq\dfrac{n}{n-1}\min_{v}\mathscr{L}_{\alpha}\left(v,v\right).
\end{equation}

Now, because $\min_{v}\mathscr{L}_{\alpha}\left(v,v\right)$ cannot
be larger than $\varDelta/\delta\leq n-1$ we have that $\rho_{\alpha,2}\leq n$
and then using the result in Corollary \ref{cor:bound RR} we have
$\mathcal{\mathcal{\mathscr{R}}_{\mathnormal{\alpha}}}\left(G\right)\geq n-1$.
The equality is obtained only for the case of the complete graph where
$\rho_{\alpha,k\neq2}=n-1$, which proves the final result.
\end{proof}
We now obtain some bounds for the Kirchhoff index for $\alpha=1$
and $\alpha=-1$, respectively.
(As usual, in the following $n$ denotes the number of vertices of $G$, while $\delta,\varDelta$ are its smallest and largest degree, respectively.)
\begin{lem}
Let $\alpha\in\{-1,1\}$. Then 
\begin{equation}
\mathcal{\mathcal{\mathscr{R}}_{\mathnormal{\alpha=1}}}\left(G\right){\color{black}{{\color{black}\geq\dfrac{n\left(n-1\right)\delta}{\varDelta^{2}}}}.}
\end{equation}
\end{lem}

\begin{proof}
We already know from Remark~\ref{rem:gers} that $\rho_{\alpha=1,n}\leq\varDelta^{2}/\delta$.
Then, $c_{\alpha=1}(v)=k_{v}\sum_{j\in\eta_{v}}k_{j}^{-1}\leq\varDelta^{2}/\delta$
and using the Corollary \ref{cor:bound RR} we obtain the result.
\end{proof}
\begin{rem}
The four graphs with the largest value of \textcolor{black}{$\mathcal{\mathcal{\mathscr{R}}_{\mathnormal{\alpha=1}}}\left(G\right)$
among all connected graphs with 8 vertices are illustrated in Fig. \ref{RR_max}.}
\end{rem}

\begin{figure}[h]
\begin{centering}
\includegraphics[width=1\textwidth]{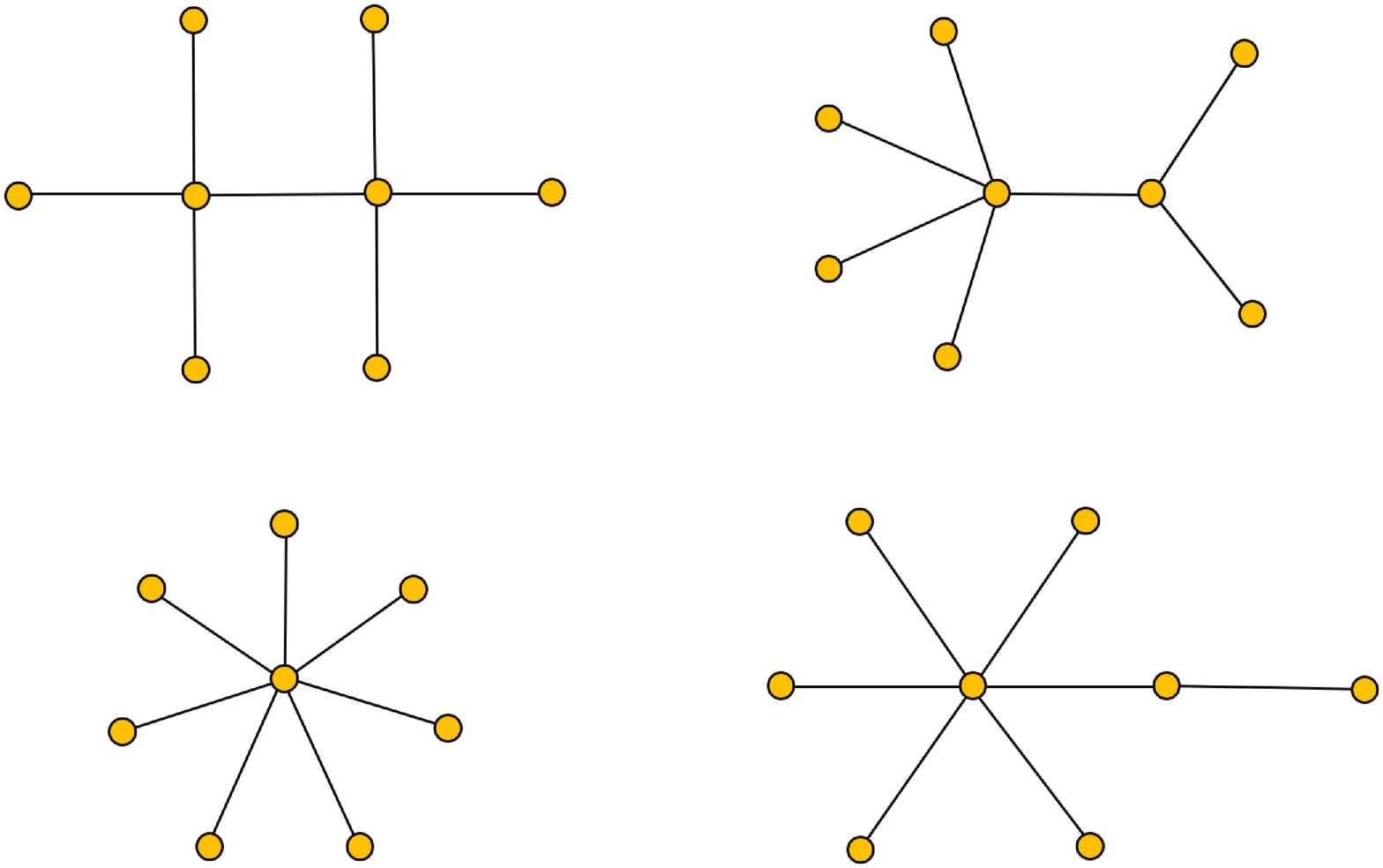}
\par\end{centering}
\caption{Graphs with the maximum values of 
$\mathcal{\mathcal{\mathscr{R}}_{\mathnormal{\alpha=1}}}\left(G\right)$
among all connected graphs with 8 vertices.}

\label{RR_max}
\end{figure}

\begin{lem}
Let $\alpha\in\{-1,1\}$. Then
\begin{equation}
\mathcal{\mathcal{\mathscr{R}}_{\mathnormal{\alpha=-1}}}\left(G\right){\color{black}{{\color{black}\geq\dfrac{n\left(n-1\right)}{2\varDelta}}}.}
\end{equation}
\end{lem}

\begin{proof}
Again by Remark~\ref{rem:gers}, 
$\rho_{\alpha=-1,n}\le 2\varDelta$. 
Then, $\underset{i}{\max}c_{\alpha=-1}\left(i\right)=\underset{i}{\max}k_{i}^{-1}\sum_{v\in\eta_{i}}k_{v}$, with the maximum being attained  at those vertices $i$ with minimal degree $k_{i}=\delta$ and connected to $\delta$ vertices $v\in\eta_{i}$
which all have the maximum degree $\varDelta$. Thus, the result follows
using the Corollary \ref{cor:bound RR}.
\end{proof}

The four graphs with the largest value of $\mathcal{\mathcal{\mathscr{R}}_{\mathnormal{\alpha=-1}}}\left(G\right)$
among all connected graphs with 8 vertices are illustrated in Fig. \ref{QA_max}.

\begin{figure}[h]
\begin{centering}
\includegraphics[width=0.8\textwidth]{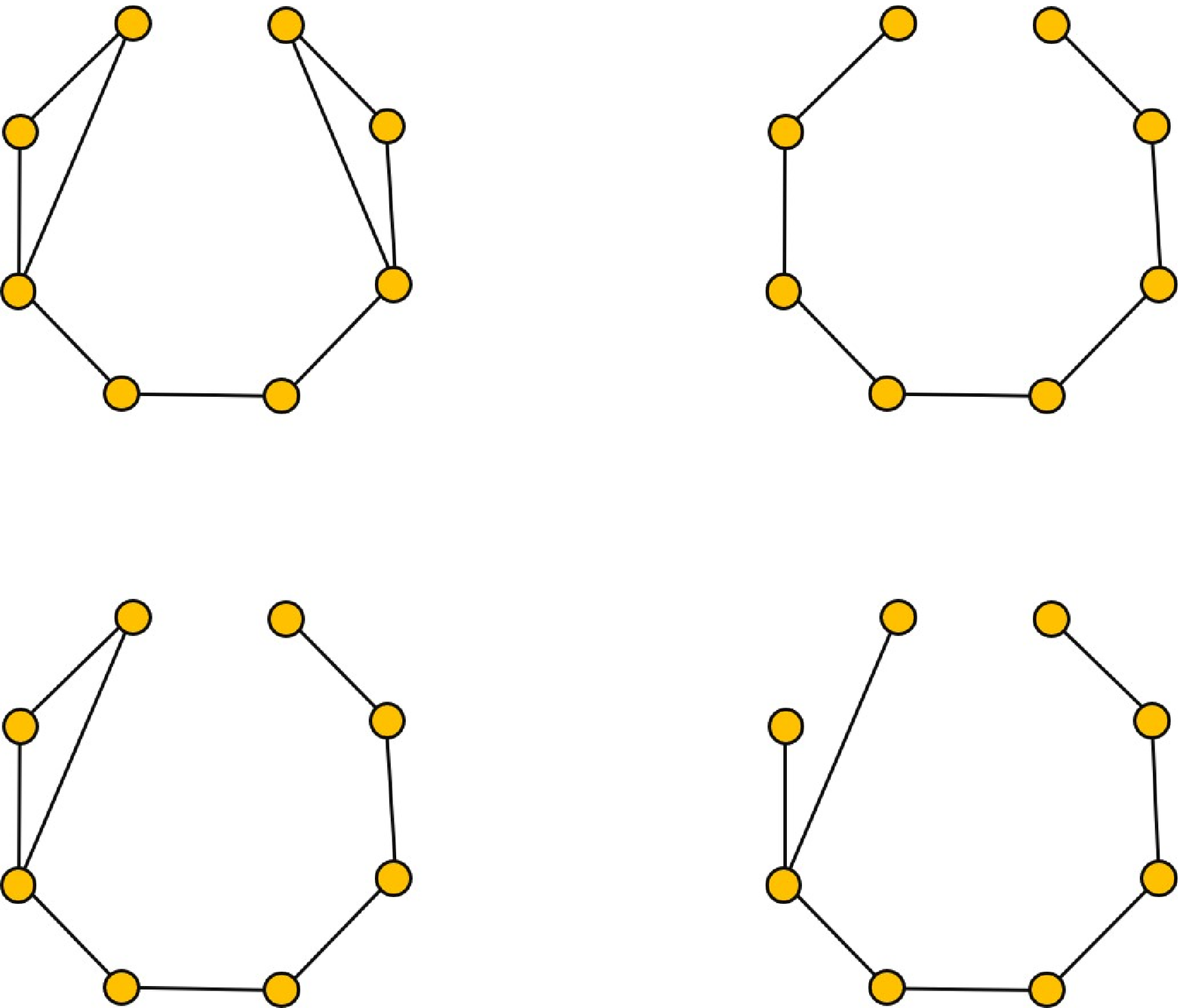}
\par\end{centering}
\caption{{Graphs with the maximum values of {$\mathcal{\mathcal{\mathscr{R}}_{\mathnormal{\alpha=-1}}}\left(G\right)$
among all connected graphs with 8 vertices.}}}\label{QA_max}
\end{figure}

\medskip{}

\begin{rem}
In $d$\textup{-regular} graphs, $c_{\alpha=1}\left(v,w\right)=c_{\alpha=-1}\left(v,w\right)=1$
for all $\left(v,w\right)\in E$. Thus, $\Omega_{\alpha=1}\left(v,w\right)=\Omega_{\alpha=-1}\left(v,w\right)=\Omega\left(v,w\right)$
for all $\left(v,w\right)\in E$, and $\mathcal{\mathcal{\mathscr{R}}_{\mathnormal{\alpha=1}}}\left(G\right)=\mathcal{\mathcal{\mathscr{R}}_{\mathnormal{\alpha=-1}}}\left(G\right)=\mathcal{\mathcal{\mathscr{R}}}_{\alpha=0}$.
\end{rem}

Then, we calculated the hubs-repelling $\mathcal{\mathcal{\mathscr{R}}_{\mathnormal{\alpha=1}}}\left(G\right)$,
attracting $\mathcal{\mathcal{\mathscr{R}}_{\mathnormal{\alpha=-1}}}\left(G\right)$
and normal $\mathcal{\mathcal{\mathscr{R}}}_{\alpha=0}\left(G\right)$
Kirchhoff indices for all connected graphs with $5\leq n\leq8$. In
general, we observed for these more than 12,000 graphs that: $\mathcal{\mathcal{\mathscr{R}}_{\mathnormal{\alpha=1}}}\left(G\right)\geq\mathcal{\mathcal{\mathscr{R}}}_{\alpha=0}\left(G\right)\geq\mathcal{\mathcal{\mathscr{R}}_{\mathnormal{\alpha=-1}}}\left(G\right)$.
Finally we formulate the following conjecture for the Kirchhoff indices
of graphs.
\begin{conj}
Let $\alpha\in\{-1,1\}$. Then
\begin{equation}
\mathcal{\mathcal{\mathscr{R}}_{\mathnormal{\alpha=1}}}\left(G\right)\geq\mathcal{\mathcal{\mathscr{R}}}_{\alpha=0}\left(G\right)\geq\mathcal{\mathcal{\mathscr{R}}_{\mathnormal{\alpha=-1}}}\left(G\right),
\end{equation}
with equality if and only if  $G$ is regular.
\end{conj}

\section{Computational results}

\subsection{Random-walks connection}

Let us consider a ``particle'' performing a standard random walk
through the vertices and edges of any of the networks studied here. The
use of random walks in graphs \cite{random walks_1} and networks
\cite{diffusion on networks} is one of the most fundamental types
of stochastic processes, used to model diffusive processes, different
kinds of interactions, and opinions among humans and animals \cite{diffusion on networks}.
They can also be used as a way to extract information about the structure
of networks, including the detection of dense groups of entities in
a network \cite{diffusion on networks}. Then, we consider a random
walk on $G$, which represents the real-world network under consideration.
We start at a vertex $v_{0}$; if at the $r$th step we are at a vertex
$v_{r}$, we move to any neighbor of $v_{r}$, with probability $k_{r}^{-1},$
where $k_{r}$ is the degree of the vertex $r$. Clearly, the sequence
of random vertices ($v_{t}:t=0,1,\ldots$) is a Markov chain \cite{random walks_1}.

An important quantity in the study of random walks on graphs is the
access or hitting time $\mathcal{H}\left(v,w\right),$ which is the
number of steps before vertex $w$ is visited, starting from
vertex $v$ \cite{random walks_1}. The sum $\mathcal{C}\left(v,w\right)=\mathcal{H}\left(v,w\right)+\mathcal{H}\left(w,v\right)$
is called the commute time, which is the number of steps
in a random walk starting at $v$, before vertex $w$ is visited and
then the walker comes back again to vertex $v$ \cite{random walks_1}.
The connection between random walks and resistance distance on graphs
is then provided by the following result (see for instance \cite{random_walk_1.5}), {where we as usual denote by $c\left(v\right)$ the vertex conductances of vertices $v$.}
\begin{lem}
For any two vertices $v$ and $w$ in $G$, the commute
time is
\begin{equation}
\mathcal{C}\left(v,w\right)=\textnormal{vol}(G)\Omega\left(v,w\right).
\end{equation}
\end{lem}

Here and in the following, $\textnormal{vol}(G)=\sum\limits_{v=1}^{n}c\left(v\right)$ is volume of the graph.
(Notice that if the graph is unweighted (the case formally corresponding to $\alpha= 0$) then by the Handshaking Lemma $\textnormal{vol}(G)=2m$.)
 The ``efficiency'' of a standard random
walk process on $G$ can then be measured by

\begin{equation}
\varepsilon\left(G\right)=1/\sum_{v,w}{\mathcal C}\left(v,w\right),
\end{equation}
That is, if a standard random walker on a graph uses small times to
commute between every pair of vertices in the graph, it is an efficient
navigational process. On the contrary, large commuting times between
pairs of vertices reveal very inefficient processes. Obviously, $\varepsilon\left(G\right)=1/\left(\textnormal{vol}(G)\sum_{v,w}\Omega\left(v,w\right)\right)$.

We now extend these concepts to the use of hubs-biased random walks
and calculated the Kirchhoff indices $\mathcal{\mathcal{\mathscr{R}}}\left(G\right),$
$\mathcal{\mathcal{\mathscr{R}}_{\mathnormal{\alpha=-1}}}\left(G\right)$
and $\mathcal{\mathcal{\mathscr{R}}_{\mathnormal{\alpha=1}}}\left(G\right)$
for all these networks. Following a similar reasoning as before we
define the efficiencies of the hubs-biased random walks by:

\begin{equation}
\varepsilon_{\alpha}\left(G\right)\coloneqq1/\left(\textnormal{vol}_{\alpha}\sum_{v,w}{\Omega}_{\alpha}\left(v,w\right)\right),
\end{equation}
where $\textnormal{vol}_{\alpha}$ is the volume of the graph with conductances based
on $\alpha$. We are interested here in the efficiency of the hubs-biased
random walk processes relative to the standard random walk. We propose
to measure these relative efficiencies by

\begin{equation}
\mathscr{E}_{\alpha}\left(G\right)\coloneqq\dfrac{\varepsilon_{\alpha}\left(G\right)}{\varepsilon\left(G\right)}=\dfrac{\textnormal{vol}}{\textnormal{vol}_{\alpha}}\dfrac{\mathcal{\mathcal{\mathscr{R}}}\left(G\right)}{\mathcal{\mathcal{\mathscr{R}}_{\mathnormal{\alpha}}}\left(G\right)}.
\end{equation}

When $\mathscr{E}_{\alpha}\left(G\right)>1$ the hubs-biased random
walk is more efficient than the standard random walk. On the other
hand, when $\mathscr{E}_{\alpha}\left(G\right)<1$ the standard random
walk is more efficient than the hubs-biased one. When the efficiency
of both processes, hubs-biased and standard, are similar we have $\mathscr{E}_{\alpha}\left(G\right)\approx1$. 

\subsection{Efficiency in small graphs}

We start by analyzing the 11,117 connected graphs with 8 vertices that
we studied previously. In Fig. \ref{All_8_vertices} we illustrate the
results in a graphical way. As can be seen
\begin{itemize}
\item $\mathscr{E}_{\alpha=-1}\left(G\right)>1$ for 95.7\% of the graphs
considered (10,640 out of 11,117), indicating that in the majority
of graphs a hubs-attracting random walk can be more efficient than
the standard random walk {driven by the unweighted Laplacian, corresponding to $\alpha= 0$};
\item $\mathscr{E}_{\alpha=1}\left(G\right)>1$ only in 91 out of the 11,117
graphs, which indicates that hubs-repelling random walks are more
efficient than the standard random walk only for very few graphs;
\item All graphs for which $\mathscr{E}_{\alpha=1}\left(G\right)\geq1$
also have $\mathscr{E}_{\alpha=-1}\left(G\right)\geq1$, with equality
only for regular graphs;
\item Only 461 graphs (4.15\% of all graphs considered) have simultaneously
$\mathscr{E}_{\alpha=1}\left(G\right)<1$ and $\mathscr{E}_{\alpha=-1}\left(G\right)<1$.
These are graphs for which the standard random walk is more efficient than both hubs-biased random walks;
\item Only 17 graphs have $\mathscr{E}_{\alpha=1}\left(G\right)=\mathscr{E}_{\alpha=-1}\left(G\right)=1$. They are all the regular graphs having 8 nodes that exist.
\end{itemize}
\begin{figure}
\begin{centering}
\includegraphics[width=0.8\textwidth]{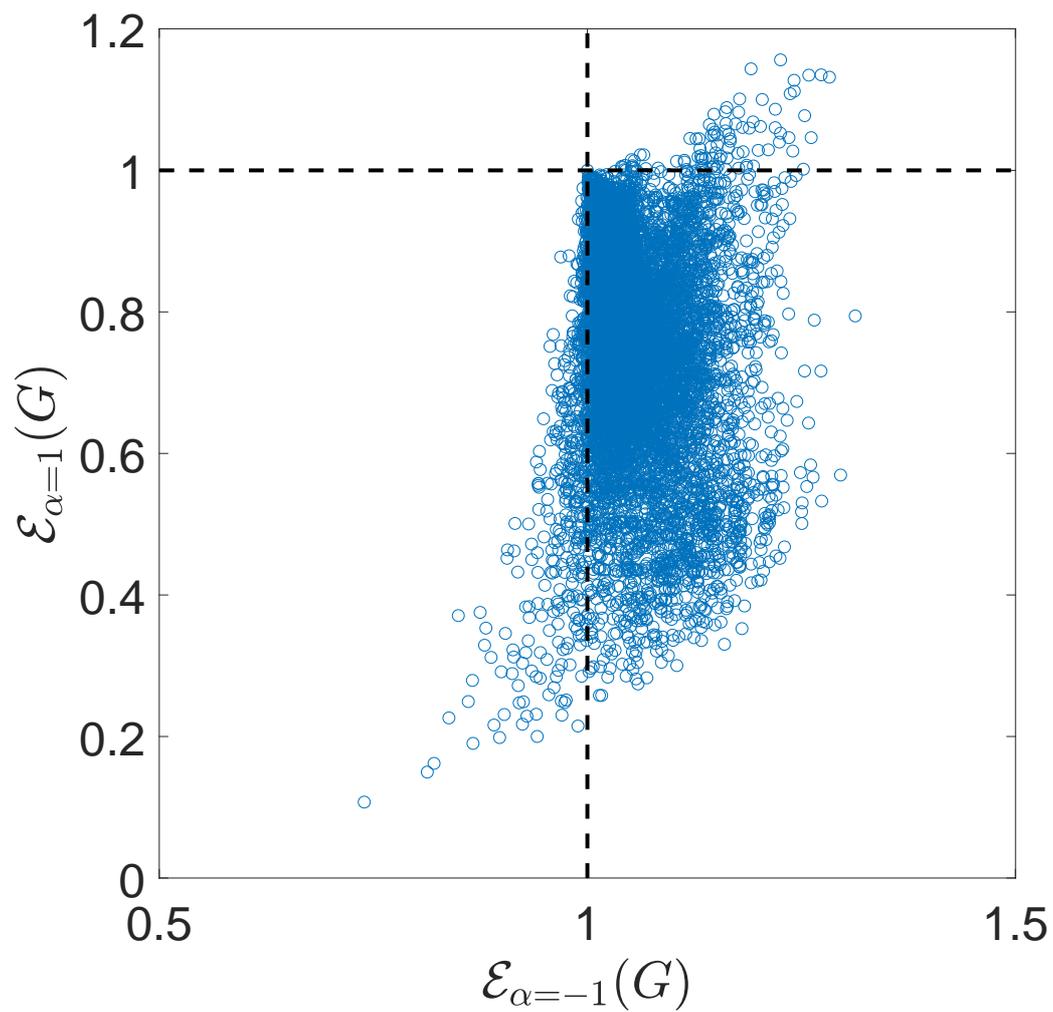}
\par\end{centering}
\caption{Plot of the efficiency of hubs-biased random walks relative to the
standard one for all 11,117 connected graphs with 8 vertices. }

\label{All_8_vertices}
\end{figure}

These results bring some interesting hints about the structural characteristics
that the graphs have to display to be benefited from hubs-biased processes.
For instance, the fact that most of the graphs can be benefited from
hubs-attracting ($\alpha=-1$) random walks is explained as follow.
A standard random walk typically does not follow the shortest topological
path connecting a pair of non-connected vertices. However, the number
of shortest paths crossing a vertex increases with the degree of that
vertex. For instance, let $k_{v}$ and $t_{v}$ be the degree and the
number of triangles incident to a vertex $v$. The number of shortest
paths connecting pairs of vertices is $P\geq\dfrac{1}{2}k_{v}\left(k_{v}-1\right)-t_{v}$.
Therefore, the hubs-attracting strategy induces the random walker
to navigate the network using many of the shortest paths interconnecting
pairs of vertices, which obviously decreases the commute time and increases
the efficiency of the process. Most networks can be benefited from
this strategy.

The case of the hubs-repelling strategies is more subtle. To reveal
the details we illustrate the four graphs having the minimum efficiency
of a hubs-repelling ($\alpha=1$) random walk in Fig. \ref{ER_minimum}.
As can be seen all these graphs have star-like structures, which are
the graphs with the largest possible degree heterogeneity \cite{heterogeneity_1,heterogeneity_2}.
Therefore, in these graphs the use of hubs-repelling strategies lets the random walker get trapped in small degree vertices without the
possibility of visiting other vertices, because for such navigation they
have to cross the hubs of the graph, a process which is impeded by
the hubs-repelling strategy.

\begin{figure}
\begin{centering}
\includegraphics[width=1\textwidth]{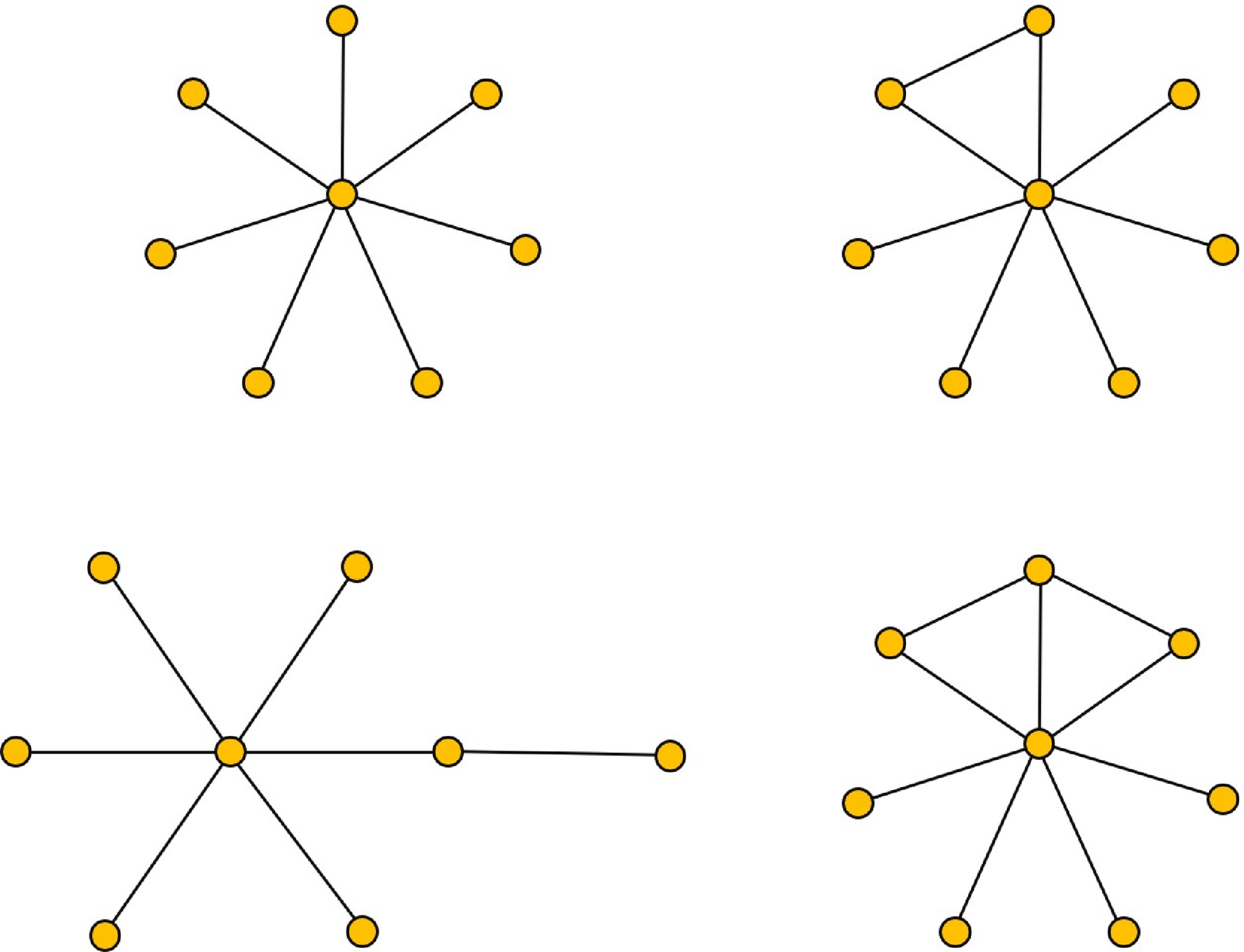}
\par\end{centering}
\caption{Illustration of the graphs with the minimum values of $\mathscr{E}_{\alpha=1}\left(G\right)$
among the 11,117 connected graphs with 8 vertices. }

\label{ER_minimum}
\end{figure}

The previous analysis allows us to consider the reason why so little
number of graphs display $\mathscr{E}_{\alpha=1}\left(G\right)>1$.
Such graphs have to display very small degree heterogeneity, but without
being regular, as for regular graphs $\mathscr{E}_{\alpha=1}\left(G\right)=1$.
The four graphs with the highest value of $\mathscr{E}_{\alpha=1}\left(G\right)$
among all connected graphs with 8 vertices are illustrated in Fig. \ref{ER_maximum}.
As can be seen these graphs display ``quasi-regular'' structures
but having at least one pendant vertex connected to another low-degree
vertex. Then, in a hubs-repelling strategy this relatively isolated
vertex (the pendant one) has larger changes (than in a standard random
walk) of being visited by the random walker who is escaping from the
vertices of larger degree.

\begin{figure}
\begin{centering}
\includegraphics[width=1\textwidth]{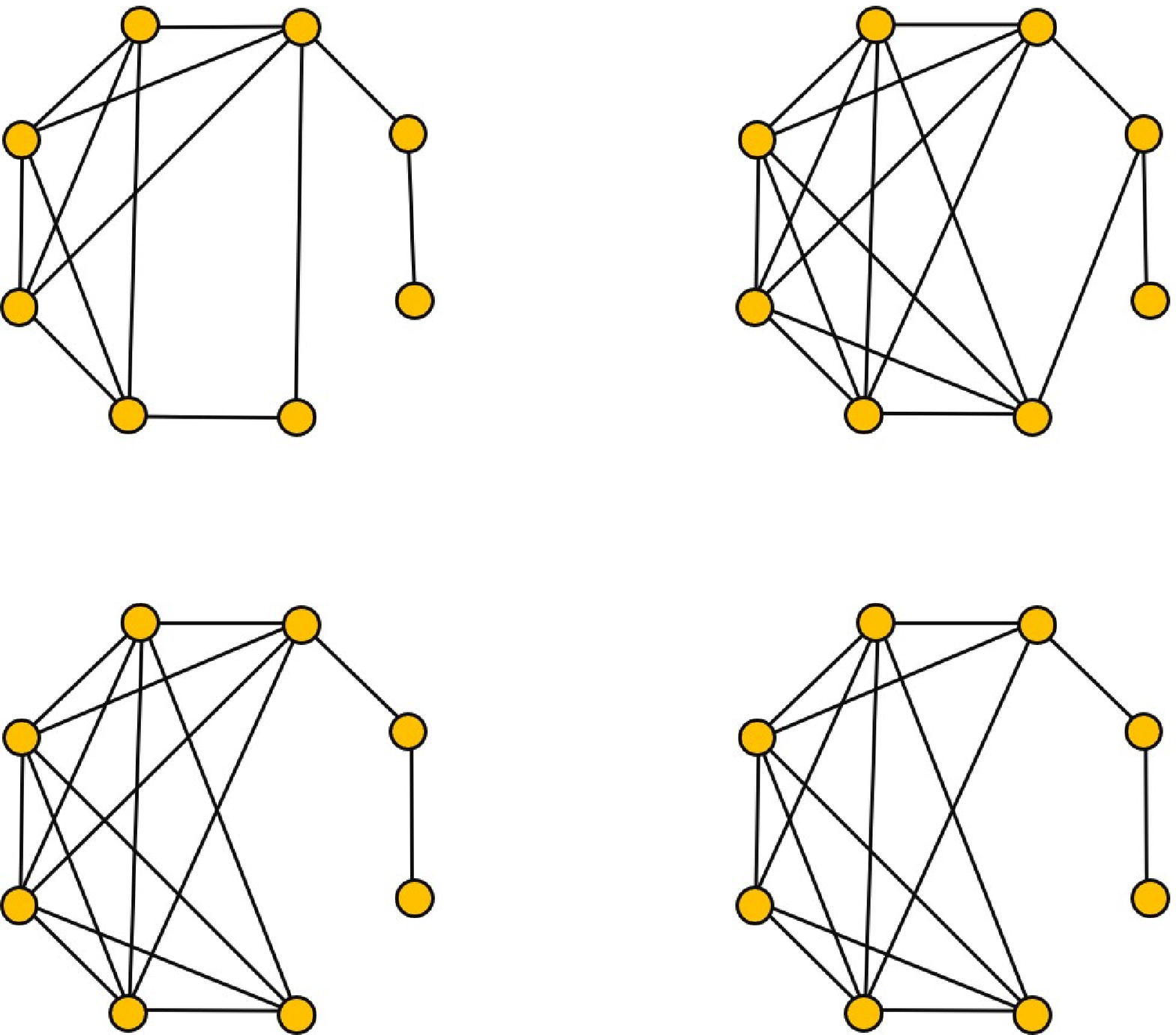}
\par\end{centering}
\caption{Illustration of the graphs with the maximum values of $\mathscr{E}_{\alpha=1}\left(G\right)$
among the 11,117 connected graphs with 8 vertices. }

\label{ER_maximum}
\end{figure}

In closing, we have observed that hubs-attracting random walks are
very efficient in most of graphs due to the fact that such processes
increase the chances of navigating the graph through their shortest
paths. On the other hand, hubs-repelling random walks are efficient
only in those quasi-regular graphs having some relatively isolated
vertices which can be visited by the hubs-repelling walker with higher
chances than in a standard random walk.

\subsection{Efficiency in real-world networks}

Here we study 59 real-world networks representing brain/neuronal systems,
electronic circuits, social systems, food webs, protein-protein interaction
networks (PIN), modular software, citation networks, transcription
networks and infrastructure networks. The description of all the networks
is in the Appendix of the book \cite{Estrada book}. 

The results obtained here for the 59 real-world networks are resumed
in Table \ref{real-world} where we report the average values of the
previous indices for groups of networks in different functional classes,
i.e., brain networks, electronic circuits, social networks, etc. 

\begin{table}
\begin{centering}
\begin{tabular}{|c|c|c|c||c|c|}
\hline 
type & number & $\mathscr{\bar{E}}_{\alpha=1}\left(G\right)$ & std & $\mathscr{\bar{E}}_{\alpha=-1}\left(G\right)$ & std\tabularnewline
\hline 
\hline 
brain & 3 & 0.6365 & 0.2552 & 0.9938 & 0.0489\tabularnewline
\hline 
circuits & 3 & 0.5187 & 0.0277 & 1.0662 & 0.0095\tabularnewline
\hline 
foodweb & 14 & 0.4134 & 0.2715 & 1.1247 & 0.2545\tabularnewline
\hline 
social & 12 & 0.3536 & 0.1917 & 1.0103 & 0.1812\tabularnewline
\hline 
citations & 7 & 0.2032 & 0.1758 & 0.8199 & 0.2769\tabularnewline
\hline 
PIN & 8 & 0.1385 & 0.0896 & 0.7855 & 0.1864\tabularnewline
\hline 
infrastructure & 4 & 0.0869 & 0.1439 & 0.4846 & 0.4638\tabularnewline
\hline 
software & 5 & 0.0712 & 0.0296 & 0.6148 & 0.1515\tabularnewline
\hline 
transcription & 3 & 0.0711 & 0.0710 & 0.5806 & 0.3218\tabularnewline
\hline 
\end{tabular}
\par\end{centering}
\caption{Average values of the relative efficiency of using a hubs-attracting
random walk on the graph with respect to the use of the standard random walk:
$\mathscr{\bar{E}}_{\alpha=-1}\left(G\right)$ for different networks
grouped in different classes. The number of networks in each class
is given in the column labeled as ``number''. The same for a hubs-repelling
random walk: $\mathscr{\bar{E}}_{\alpha=1}\left(G\right)$. In both
cases the standard deviations of the samples of networks in each class
is also reported.}

\label{real-world}
\end{table}

The main observations from the analysis of these real-world networks
are:
\begin{itemize}
\item $\mathscr{E}_{\alpha=-1}\left(G\right)>1$ for 50.8\% of the networks
considered, indicating that only in half of the networks a hubs-attracting
random walk can be more efficient than the standard random walk;
\item $\mathscr{E}_{\alpha=1}\left(G\right)>1$ in none of the networks,
which indicates that standard random walks are always more efficient
than hubs-repelling random walks in all these networks.
\end{itemize}
The first result is understood by the large variability in the degree
heterogeneity of real-world networks. In this case, only those networks
with skew degree distributions are benefited from the use of hubs-attracting
random walks, while in those with more regular structures are not. 

The second result indicates that there are no network with such quasi-regular
structures where some of the vertices are relatively isolated as the
graphs displayed in Fig. \ref{ER_maximum}. However, as usual the
devil is in the details. The analysis of the results in Table \ref{real-world}
indicates that brain networks, followed closely by electronic flip-flop
circuits, are the networks in which the use of hubs-repelling strategies
of navigation produces the highest efficiency relative to standard
random walks. This can also be read as that these brain networks have
evolved in a way in which their topologies guarantee random walk processes
as efficient as the hubs-attracting ones without the necessity of
navigating the brain using such specific mechanisms. In addition,
the use of hubs-repelling processes do not affect significantly the
average efficiency of brain networks, as indicated by the value of
$\mathscr{\bar{E}}_{\alpha=1}\left(G\right),$ which is very close
to one. This result indicates that if these networks have to use a
hubs-repelling strategies of navigation due to certain biological
constraints, they have topologies which are minimally affected  --  in
terms of efficiency  --  when using such strategies.

Finally, another remarkable result is that the efficiency of navigational
processes in infrastructural and modular software systems is more
than 1000\% efficient by using normal random-walk approaches than
by using hubs-repelling strategies. Those infrastructural networks
seem to be wired to be navigated by using their hubs, and avoiding
them cost a lot in terms of efficiency. This is clearly observed in
many transportation networks, such as air transportation networks,
where the connection between pairs of airports is realized through
the intermediate of a major airport, i.e., a hub, in the network.

\section{Conclusions}

We have introduced the concept of hubs-biased resistance distance.
These Euclidean distances are based on graph Laplacians which consider
the edges $e=\left(v,w\right)$ of a graph weighted by the degrees
of the vertices $v$ and $w$ in a double orientation of that edge.
Therefore, the hubs-biased Laplacian matrices are non-symmetric and
reflect the capacity of a graph/network to diffuse particles using
hubs-attractive or hubs-repulsive strategies. The corresponding hubs-biased
resistance distances and the corresponding Kirchhoff indices can be
seen as the efficiencies of these hubs-attracting/repelling random
walks of graphs/networks. We have proved several mathematical results
for both the hubs-biased Laplacian matrices and the corresponding
resistances and Kirchhoff indices. Finally we studied a large number
of real-world networks representing a variety of complex systems in
nature, and society. All in all we have seen that there are networks
which have evolved, or have being designed, to operate efficiently
under hubs-attracting strategies. Other networks, like brain ones,
are {almost immune to the change of strategies,
because the use of hubs-attracting strategies improve very little
the efficiency of a standard random walk, and the efficiency of hubs-repelling
strategies is not significantly different than that of the classical
random walks. Therefore, in such networks the use of the standard
random walk approach is an efficient strategy of navigation, while
infrastructures and modular software networks seem to be designed
to be navigated by using their hubs.}

\section*{Acknowledgment}

The work of D.M. was supported by the Deutsche Forschungsgemeinschaft
(Grant 397230547). E.E. thanks financial support from Ministerio de
Ciencia, Innovacion y Universidades, Spain for the grant PID2019-107603GB-I00
``Hubs-repelling/attracting Laplacian operators and related dynamics
on graphs/networks''. This article is based upon work from COST Action 18232 MAT-DYN-NET, supported by COST (European Cooperation in Science and Technology), www.cost.eu.


\end{document}